\documentclass[10pt]{article}

\usepackage{bm}
\usepackage[top=0.9in,bottom=1.1in,left=1.05in,right=1.05in]{geometry}

\usepackage{amsmath}
\usepackage{graphicx}
\usepackage[colorinlistoftodos]{todonotes}
\usepackage[colorlinks=true, allcolors=blue]{hyperref}
\usepackage{amsmath}
\usepackage{upgreek}
\usepackage{amssymb}
\usepackage{amsfonts}
\usepackage{amsthm}
\usepackage{mathrsfs}
\usepackage{fontenc}
\usepackage{empheq}
\usepackage{enumerate}
\usepackage{comment}
\usepackage{appendix}
\usepackage{mathtools}
\mathtoolsset{showonlyrefs}

\DeclareMathOperator*{\argmin}{arg\,min}

\theoremstyle{plain}
\newtheorem{thm}{Theorem}
\newtheorem{lem}{Lemma}
\newtheorem{prop}[thm]{Proposition}
\newtheorem{cor}{Corollary}
\newtheorem{assu}{Assumption}

\newcommand{\bX}{\bm{X}}
\newcommand{\bx}{\bm{x}}
\newcommand{\bY}{\bm{Y}}
\newcommand{\bZ}{\bm{Z}}

\newcommand{\bW}{\bm{W}}
\newcommand{\balpha}{\bm{\alpha}}
\newcommand{\ZZ}{\mathcal{Z}}
\newcommand{\RR}{\mathbb{R}}
\newcommand{\EE}{\mathbb{E}}
\newcommand{\QQ}{\mathbb{Q}}
\newcommand{\PP}{\mathbb{P}}
\newcommand{\MCA}{\mathcal{A}}

\newcommand{\ud}{\,\mathrm{d}}
\newcommand{\mc}[1]{\mathcal{#1}}
\newcommand{\half}{\frac{1}{2}}
\newcommand{\cN}{\mathcal{N}}
\newcommand{\miniH}{\check {\bm H}}
\newcommand{\pinorm}{\|\pi\|}
\newcommand{\eps}{\epsilon}

\theoremstyle{definition}
\newtheorem{defn}{Definition}[section]

\theoremstyle{remark}
\newtheorem{rem}{Remark}

\newcommand{\jh}[1]{\textcolor{red}{\textsf{[JH: #1]}}}

\newcommand{\bR}{\mathbb {R}}
\newcommand{\transpose}{^{\operatorname{T}}}
\newcommand{\rmd}{\,\mathrm{d}}
\title{Convergence of Deep Fictitious Play for Stochastic Differential Games}
\author{Jiequn Han\thanks{Department of Mathematics, Princeton University, Princeton, NJ 08544-1000, USA, \em{jiequnh@princeton.com}.} \and Ruimeng Hu\thanks{Department of Mathematics, and Department of Statistics and Applied Probability, University of California, Santa Barbara, CA 93106-3080, {\em rhu@ucsb.edu}. RH was partially supported by the NSF grant DMS-1953035.} \and  Jihao Long\thanks{The Program in Applied and Computational Mathematics, Princeton University, Princeton, NJ 08544-1000, \em{jihaol@princeton.edu}.}}
\date{\today}

\begin{document}

\maketitle

\begin{abstract}

Stochastic differential games have been used extensively to model agents' competitions in Finance, for instance, 
in P2P lending platforms from the Fintech industry, the banking system for systemic risk, and insurance markets. The recently proposed machine learning algorithm, deep fictitious play, provides a novel efficient tool for finding Markovian Nash equilibrium of large $N$-player asymmetric stochastic differential games [J. Han and R. Hu,  Mathematical and Scientific Machine Learning Conference, pages 221-245, PMLR, 2020]. By incorporating the idea of fictitious play, the algorithm decouples the game into $N$ sub-optimization problems, and identifies each player's optimal strategy with the deep backward stochastic differential equation (BSDE) method parallelly and repeatedly. In this paper, we prove the convergence of deep fictitious play (DFP) to the true Nash equilibrium. We can also show that the strategy based on DFP forms an $\eps$-Nash equilibrium.
We generalize the algorithm by proposing a new approach to decouple the games, and present numerical results of large population games showing the empirical convergence of the algorithm beyond the technical assumptions in the theorems.

\end{abstract}


\section{Introduction}

Deep neural network has become a popular and powerful tool in scientific computing, for its remarkable performance in approximating high-dimensional functions. Its success has brought natural applications in stochastic differential games, where high-dimensional optimization problems and/or stochastic differential equations are solved to provide the modeling and analysis of tactical interactions among multiple decision-makers in the context of a random dynamical system. These decision-makers, usually referred to as players or agents, can interact in a manner ranging from completely non-cooperative to completely cooperative. The nature of uncertainty makes stochastic differential games appropriate to be used for the study of competitions in Finance, {\it e.g.}, in P2P lending platforms \cite{WeLi:17,LiQiWaLi:19} from the Fintech industry and insurance markets \cite{Ze:10,BeSiYaYa:14,chYaZe:18}.


For non-cooperative stochastic differential games, a core problem is to compute the associated Nash equilibrium, which refers to a set of strategies so that when applied, no player will profit from unilaterally changing her own choice. When the games involve heterogeneous agents of moderate size, {\it e.g.}, $5 \leq N \leq 100$, computing the Nash equilibrium becomes numerically challenging since conventional numerical algorithms lose their efficiency for $N$ beyond $5$, and the mean-field framework has not started to work well with $N \leq 100$ asymmetric players.


To address the challenge, the authors have recently proposed the deep fictitious play (DFP) algorithms \cite{HaHu:19}, providing a novel efficient tool for finding Markovian Nash equilibrium of large $N$-player asymmetric stochastic differential games. However, despite the efficient performance in simulation, the algorithm's theoretical foundation is still lacking, which will be the focus of this paper. In addition, we generalize the previous algorithms, and propose a general two-step scheme: The first step aims to recast the game into $N$ sub-problems that will be repeatedly solved. The desired algorithm requires that, after the recast, the sub-problems are decoupled among different players given the previous stage's solutions, and that their solutions converge to the true Nash equilibrium. Specifically, we propose two options for the first step: 
\begin{enumerate}
    \item[I.] Fictitious play. This approach was used in \cite{HaHu:19}, assuming that players are myopic and will choose their best responses against others' previous stage action at every subsequent stage. Therefore each player still faces a nonlinear optimization problem.
    \item[II.] Policy update. This approach calculates the game values using all responses from the previous stage, and the current stage responses are determined as if they are the optimizers of the calculated game values.
\end{enumerate}
The second step of the DFP algorithm aims to solve the sub-problems efficiently and accurately. Remark that, due to the large number of players and the high dimensionality of the controlled state process, each sub-problem may still be high-dimensional after the decoupling step. In \cite{HaHu:19}, the Deep BSDE method was employed for each sub-problem, which presents excellent performance. It relies on the BSDE representations of semi-linear partial differential equations (PDEs) and deep learning approximations after discretizing the BSDE by an Euler scheme. It parametrizes the initial position of the backward process and the adjoint process by DNNs, then simulates both processes in a forward manner, aiming to minimize the discrepancy between the terminal value of the backward process and its network approximation. The analysis for the second step shall focus on this method. Meanwhile, we remark that other deep neural networks (DNNs) based algorithms, such as deep learning backward dynamic programming (DBDP) method \cite{HuPhWa:20} and deep Galerkin method \cite{SiSp:18}, are also promising choices for solving sub-problems.




\medskip
\noindent\textbf{Related literature.} The theoretical study of differential games was initiated by R. Isaacs in the early 1960s \cite{Is:99}. Later on, to better describe read world's uncertainties, noises are added to the state of the system, and stochastic differential games now have been intensively used across many disciplines. Domains of applications include management science ({\it e.g.}, operations management, marketing, finance, systemic risk), economics ({\it e.g.}, industrial organization, environmental and macroeconomics, production of exhaustible resources), social science ({\it e.g.}, networks, crowd behavior, congestion), biology ({\it e.g.}, flocking), and military ({\it e.g.}, cyber-attacks).

Fictitious play is well documented in the economics literature, as a learning process for finding Nash equilibria. It was firstly proposed by \cite{Br:49,Br:51} for normal-form games.
Since then, there have been extensive studies on the convergence of fictitious play or its variation under different setting, for instance, see \cite{MiRo:91,MoSh:96,KrSj:98,HoSa:02,Be:05}. For stochastic differential games, besides \cite{HaHu:19}, the most related work is \cite{Hu2:19}, where fictitious play is used to design numerical algorithms for finding open-loop Nash equilibria. We remark that, the idea of fictitious play is not limited to study the games with a moderate number of heterogeneous players \cite{Hu2:19,HaHu:19}, but has also been applied in mean-field games, {\it e.g.}, see \cite{CaHa:17,BrCa:18,El:19}.

The proposed policy update for the first step of the DFP algorithm closely follows policy iteration (PI) in spirit, which was initially introduced by Howard \cite{Ho:60} for discounted Markovian decision problems (MDP). It consists of two steps: policy evaluation (obtaining the expected reward for a given policy) and policy improvement (updating the policy using the rewards for successor states). PI was later generalized to modified PI in \cite{Pu:94}, and has remained as the method of choice in designing reinforcement learning algorithms, {\it e.g.}, see \cite{Go:04,PoMa:11} and the references therein. 

The literature of using DNNs for learning high-dimensional function is rich, including methods for solving high-dimensional parabolic PDEs and BSDEs ({\it e.g.}, the deep BSDE method \cite{EHaJe:17,HaJeE:18}, the DBDP \cite{HuPhWa:20,germain2020deep}, and many others \cite{SiSp:18,beck2019deep,BeEJe:19,PhWaGe:19,YuXiSu:19,JiPePeZh:20}). It also yields algorithms for solving the Schr\"{o}dinger equation~\cite{HaZhE:19,Pfau2019abinitio,han2020solving}, stochastic control problems~\cite{HaE:16,nakamura2019adaptive}, mean field games~\cite{CaNiJa:19,AnFoLa:20} and nonlinear optimal stopping problems~\cite{Hu:19}.

\medskip
\noindent\textbf{Main contribution.} The contribution of this paper consists of the following: 1. We propose a general two-step scheme that extends the original deep fictitious play algorithm \cite{HaHu:19}, and provide two options for solving the first step. The proposed algorithm can efficiently solve stochastic differential games with heterogeneous agents of large size ({\it e.g.}, $5 \leq N \leq 100$), and the presence of common noise. 2. We provide the theoretical foundation for the proposed algorithms. In specific, with small time duration, we prove that the solutions to the decoupled sub-problems, if solved repeatedly and exactly at each stage, converge to the true Nash equilibrium; that the numerical solutions to each sub-problem tend to be exact as we refine the time step in the Euler scheme; and that the strategy based on numerical solutions forms an $\eps$-Nash equilibrium, after running sufficiently many stages and using sufficiently fine time step. 3. We present numerical results showing empirical convergence even beyond the technical assumptions used in the theorems.

\medskip

The rest of this paper is organized as follows. In Section~\ref{sec_mathformulation}, we give the mathematical formulation of general $N$-player asymmetric stochastic games in continuous time. The algorithms consisting of the decoupling step and sub-problem-solving step via deep learning are detailed
in Section~\ref{sec_algorithm}. Section~\ref{sec_analysis} provides convergence analysis for the proposed algorithms, followed by numerical examples presented in Section~\ref{sec_numerics}. We make conclusive remarks in Section~\ref{sec_conclusion}.


\section{Mathematical Formulation}\label{sec_mathformulation}
Throughout the paper, we shall use the following notations:
\begin{itemize}
    \item A boldface character refers to a collection of objects from all players;
    \item A regular character with a superscript $i$ refers to an objective from player $i$ (no matter a scalar or a vector) or the $i^{th}$ column of a vector;
    \item A boldface character with a superscript $-i$ refers to a collection of objects from all players except $i$;
    \item The state process $\bX_t$ introduced below is a common process to all players, and will always be in boldface.
\end{itemize}


We consider a general $N$-player non-zero-sum stochastic differential games. An $\RR^n$-valued \emph{common} state process $\bX_t$ is controlled by a Markovian strategy/policy\footnote{Hereafter, we shall use \textit{strategy} and \textit{policy} interchangeably.} $\balpha$:
\begin{equation}\label{def_Xt}
\ud \bm X_t^{\bm \alpha} = b(t, \bm X_t^{\bm \alpha}, \bm \alpha(t, \bm X_t^{\bm \alpha})) \ud t + \Sigma(t, \bm X_t^{\bm \alpha}) \ud \bm W_t, \quad \bm X_0 = \bm x_0,
\end{equation}
where $\balpha = (\alpha^1, \ldots, \alpha^N)$ is the collection of all players' $\MCA^i$-valued strategies. 
For simplicity, we assume $\MCA^i = \bR^{d_{\alpha}}$ for $i = 1, 2, \dots, N$. 
If not ({\it e.g.}, some boundedness constraints are put on $\alpha^i$), we can assume there exist Lipschitz mappings $P_\alpha^i$ from $\bR^{d_\alpha}$ to $\mathcal{A}^i$ so that $\mathcal{A}^i = P_\alpha^i(\bR^{d_\alpha})$, and all the statements below hold easily with the help of the Lipschitz mappings.
The drift and diffusion coefficients $b$ and $\Sigma$ are deterministic functions of the common state, $b\colon [0,T] \times \RR^n \times \mc{A}  \hookrightarrow \RR^n$,  $\Sigma\colon [0,T] \times \RR^n \hookrightarrow \RR^{n \times k}$, where $\mathcal{A} =  \otimes_{i=1}^N \mc{A}^i = \bR^{Nd_\alpha}$ is the space for the joint control $\balpha$, and $\bm W$ is a $k$-dimensional standard Brownian motion on a filtered probability space $(\Omega,\mathbb{F},\{\mathcal{F}_t\}_{0 \le t \le T}, \mathbb{P})$.

Player $i$ aims at minimizing her expected total cost:
\begin{equation}\label{def_obj}
    \inf_{\alpha^i \in \mathbb{A}^i}\EE\left[\int_0^T f^i(s, \bm X_s^{\bm \alpha}, \bm \alpha(s,\bm  X_s^{\bm \alpha})) \ud s + g^i(\bm X_T^{\bm \alpha}) \right]
\end{equation}
by choosing $\alpha^i$ among all admissible strategies $\mathbb{A}^i$:
\begin{equation}
    \mathbb{A}^i = \left\{\alpha^i(t, \bx): \text{ Borel measurable function } [0,T] \times \RR^n \hookrightarrow \bR^{d_\alpha}
    \right\},
\end{equation}
where the running cost $f^i\colon [0,T] \times \RR^{n} \times \mathcal{A} \hookrightarrow \RR$ and the terminal cost $g^i\colon \RR^{n} \hookrightarrow \RR$ are deterministic measurable functions.
Obviously, the quantity in \eqref{def_obj} is also affected by other players' strategies $\alpha^j$. To emphasis this dependence, we introduce the notation $J^i_t(\alpha^1, \ldots, \alpha^N)$ for the cost of player $i$ starting at $t$ when players choose their strategies $(\alpha^1, \ldots, \alpha^N)$:
\begin{equation}\label{def_J}
J^i_t(\alpha^1, \ldots, \alpha^N) \equiv J_t^i(\bm \alpha) := \EE\left[\int_t^T f^i(s, \bm X_s^{\bm \alpha}, \bm \alpha(s,\bm  X_s^{\bm \alpha})) \ud s + g^i(\bm X_T^{\bm \alpha}) \right].
\end{equation}

In the following sections, we shall present the algorithms for solving the above game and prove its theoretical convergence. In particular, we are interested in finding a Markovian Nash equilibrium (or the Markovian $\eps$-Nash equilibrium).

\begin{defn}
A Markovian $\eps$-Nash equilibrium is a tuple $\bm \alpha^\eps = (\alpha^{1, \eps}, \ldots, \alpha^{N, \eps}) \in \mathbb{A}$, such that, for non-negative $\eps$,
\begin{equation}
\forall i \in \mc{I}, \text{ and } \alpha^i \in \mathbb{A}^i, \quad J^i_0(\bm \alpha^\eps) - \eps \leq J^i_0(\alpha^{1, \eps}, \ldots, \alpha^{i-1, \eps}, \alpha^i, \alpha^{i+1, \eps}, \ldots, \alpha^{N, \eps}).
\end{equation}
A Markovian Nash equilibrium, denoted by $\balpha^\ast$,  is equivalent to an $\eps$-Nash equilibrium where $\eps = 0$. Here $\mathbb{A} = \otimes_{i=1}^N \mathbb{A}^i$ is the product space of $\mathbb{A}^i$, and $\mc{I} = \{1, 2, \ldots, N\}$ is the set of all players.
\end{defn}

As discussed in \cite{HaHu:19}, the formulation \eqref{def_Xt}--\eqref{def_obj} is less restrictive than the usual case where player $i$ can only control her \emph{private} state. Here, a common state $\bX_t$ is controlled by all agents, which is a common feature in economics literature (see {\it e.g.}, \cite{DoJoVaSo:00,PrSe:04,Va:11}). Therefore, it is important to include it in our framework, although this will increase the coupling and make the problem harder to solve, both theoretically and numerically. Remark that the difficulty still persists in the limiting problem as $N \to \infty$ with indistinguishable players, when allowing $\alpha^i$ entering into others' states. This is called the extended mean-field game and it has attracted certain attention recently ({\it e.g.}, \cite{GoPaVo:14,GoVo:16,CaLe:18}). On the other hand, by choosing $b$ and $\Sigma$ in \eqref{def_Xt} properly, one can reduce the formulation \eqref{def_Xt} to the simpler case where each player controls her private state through $\alpha^i$. For instance, if each player's private state is $d$-dimensional, we can let $n = dN$, $b = (b^1, \ldots, b^\ell, \ldots, b^n)$ with
$b^\ell \equiv b^\ell(t, \bm x, \alpha^i)$ for $\ell = (i-1)d+1, \ldots, id$,
then the problem \eqref{def_Xt}--\eqref{def_obj} is the standard modeling in literature in many disciplines including social science, management science and engineering, with the $i^{th}$ player's $d$-dimensional private state $(X_t^{(i-1)d+1}, \ldots, X_t^{id})$ controlled by $\alpha^i$ only.

In the Markovian setting, the value function of player $i$ reads as:
\begin{equation}
V^i(t, \bm x) = \inf_{\alpha^i \in \mathbb{A}^i} \EE\left[\int_t^T f^i(s, \bm X_s^{\bm \alpha}, \bm \alpha(s,\bm  X_s^{\bm \alpha})) \ud s + g^i(\bm X_T^{\bm \alpha}) \Big \vert \bm X_t^{\balpha} = \bm x\right].
\end{equation}
Then, to compute the Markovian Nash equilibrium, we apply the dynamic programming principle and obtain a system of Hamilton-Jacobi-Bellman (HJB) equations:
\begin{align}
\label{def_HJB}
\begin{dcases}
V_t^i + \inf_{\alpha^i \in \mc{A}^i} \left\{ b(t, \bx, \balpha) \cdot \nabla_{\bx}V^i + f^i(t, \bx, \balpha)\right\}  + \half \text{Tr}(\Sigma\transpose \text{Hess}_{\bx} V^i \Sigma) = 0,\\
V^i(T,\bx) = g^i(\bx), \quad i \in \mc{I},
\end{dcases}
\end{align}
where $V_t^i$, $\nabla_{\bx} V^i,~\text{Hess}_{\bx} V^i$ denote the derivative of $V^i$ with respect to $t$, the gradient and the Hessian of function $V^i$ with respect to $\bm x$, and $\text{Tr}$ denotes the trace of a matrix. Note that the HJB system~\eqref{def_HJB} is coupled, as each minimizer $\alpha^{i, \ast}$ depends on $V^i$ and the function $b$ and $f^i$ in \eqref{def_HJB} depend on all minimizers $\balpha^\ast = (\alpha^{1, \ast}, \ldots, \alpha^{N, \ast})$.

Under appropriate conditions, the solution to \eqref{def_HJB} is related to BSDEs, using nonlinear Feynman-Kac formula (cf. \cite{PaPe:92,ElPeQu:97,PaTa:99}). To ease our notations of the BSDEs, 
we prescribe the following the relation on $b $ and $\Sigma$.

\begin{assu}\label{assumption bsigma}
There exists a measurable function $\phi$: $[0,T]\times \bR^n \times \MCA \to \bR^k$, so that $\Sigma(t,\bm{x}) \phi(t,\bm{x},\bm{\alpha}) = b(t,\bm{x},\bm{\alpha})$ for any $(t,\bm{x},\bm{\alpha}) \in [0,T]\times \bR^n \times \mathcal{A}$.
\end{assu}

Consequently, we can define the Hamiltonian function $\bm H(t, \bx, \balpha, \bm p): [0,T] \times \RR^n \times \mathcal{A} \times \RR^{k \times N} \to \RR^N$ by:
\begin{equation}\label{def_Hamiltonian}
    \bm H = [H^1, \ldots, H^N]\transpose, \quad H^i(t, \bx, \balpha, p^i) = \phi(t, \bx, \balpha) \cdot  p^i + f^i (t, \bx, \balpha),  
\end{equation}
where $p^i$ denotes the $i^{th}$ column of $\bm p$, and thus is an $\RR^{k}$ vector. Using this notation, the HJB system can be rewritten as:
\begin{equation}\label{eq_V}
    V_t^i + \inf_{\alpha^i \in \MCA^i} H^i(t, \bx, \balpha, \Sigma\transpose \nabla_{\bx} V^i) + \half \text{Tr}(\Sigma\transpose \text{Hess}_{\bx} V^i \Sigma) = 0, \quad \forall i \in \mc{I}.
\end{equation}
To better describe the optimal game policies, we define $\bm a(t, \bx, \balpha, \bm p):  [0,T] \times \RR^n \times \mathcal{A} \times \RR^{k \times N} \to \MCA$ by: 
\begin{equation}\label{def_a}
    \bm a = (a^1, \ldots, a^N), \quad  a^i(t, \bx, \balpha^{-i}, p^i) = \argmin_{\alpha^i \in \MCA^i} H^i(t, \bx, (\alpha^i, \balpha^{-i}), p^i),  \;\forall i \in \mc{I}.
\end{equation}
In other words, $a^i$ is the minimizer of the $i^{th}$ Hamiltonian, with an emphasis of the dependence on the $i^{th}$ player's game value $\Sigma\transpose \nabla_{\bx}V^i$ and others' strategies $\balpha^{-i}$. Then, we define a function $\balpha(t, \bx, \bm p)$ as the fixed point of   
\begin{equation}\label{def_alphagameast}
   \balpha = \bm a(t, \bx, \balpha, \bm p).
\end{equation}
Note that, with the above notations $\bm a$ and $\balpha$, we have assumed the minimizer in \eqref{def_a} exists and is unique, and \eqref{def_alphagameast} has a unique fixed point. Later in Assumption~\ref{assumption 1 revise}, we will detail explicit conditions on the model parameters, such that these assumptions are satisfied.

We now state the corresponding BSDE formulation of \eqref{def_HJB}, which is the key component of the algorithm design in Section~\ref{sec_algorithm} and the convergence analysis in Section~\ref{sec_analysis}. Let $(\bX_t, \bm Y_t, \bm Z_t) \in \RR^n \times \RR^N \times \RR^{ k \times N}$ be the solution to the following BSDE:
\begin{align}\label{def_BSDE}
\begin{dcases}
\bX_t =  \bx_0 + \int_0^t \Sigma(s, \bX_s) \ud \bm W_s,  \\
\bm Y_t = \bm g(\bX_T) + \int_t^T \miniH(s, \bX_s, \bm Z_s) \ud s - \int_t^T \bm Z_s\transpose \ud \bm W_s,
\end{dcases}
\end{align}
where $\miniH(t, \bx, \bm p) := \bm H(t, \bx, \balpha(t, \bx, \bm p), \bm p)$ is the minimized Hamiltonian vector, and $\bm g(\bx) \equiv [g^1, \ldots, g^N]\transpose(\bx)$ is the vector form of all terminal costs. Then we have the relation:
\begin{align}
    &\bm Y_t = [Y_t^1, \ldots, Y_t^N]\transpose, \quad Y_t^i = V^i(t, \bX_t), \\
    &\bm Z_t = [Z_t^1, \ldots, Z_t^N],\quad  Z_t^i = \Sigma\transpose(t, \bX_t) \nabla_{\bx} V^i(t, \bX_t),
    \label{eq:FK relation}
\end{align}
and the optimal game policy is expressed by
\begin{equation}
    \balpha^\ast_t = \balpha(t, \bX_t, \bZ_t).
\end{equation}
Using the relation \eqref{eq:FK relation}, we notice $  \balpha^\ast_t= \balpha(t, \bX_t, \Sigma\transpose(t, \bX_t) \nabla_{\bx}\bm V(t, \bX_t))$\footnote{We use $\nabla_{\bx} \bm V$ as an $n \times N$ matrix.}, and sometimes write $\balpha^\ast_t = \balpha^\ast(t, \bX_t)$.

\begin{rem}
Note that the process $\bX_t$ in \eqref{def_BSDE} does not allude to $b = 0$ in the controlled dynamics $\bX_t^{\balpha}$ defined in \eqref{def_Xt}. Indeed, it is an auxiliary forward stochastic process derived from the HJB system (6) using the nonlinear Feyman-Kac formula, which is an object different from the controlled process ${\bX}_t^{\balpha}$ in equation \eqref{def_Xt}. One, of course, has the flexibility to choose a different forward process with nonzero drift:
\begin{equation}
\begin{dcases}
  \tilde \bX_t = \bx_0 + \int_0^t  \Sigma(s, \tilde \bX_s) \mu(s, \tilde \bX_s)\ud s + \int_0^t \Sigma(s, \tilde \bX_s) \ud \bW_s,\\
    \tilde {\bm Y}_t = \bm g(\tilde \bX_T) + \int_t^T \miniH(s, \tilde \bX_s, \tilde {\bm Z}_s) - \tilde{\bm Z}_s\transpose \mu(t, \tilde \bX_s) \ud s - \int_t^T \tilde{\bm Z}_s\transpose \ud \bm W_s,
\end{dcases}
\end{equation}
and to express the solution to \eqref{eq_V} via \eqref{eq:FK relation} with all $(X, Y, Z)$ replaced by $(\tilde X, \tilde Y, \tilde Z)$. This is essentially rewriting equation \eqref{eq_V} to 
\begin{equation}
    V_t^i + \inf_{\alpha^i \in \MCA^i} H^i(t, \bx, \balpha, \Sigma\transpose \nabla_{\bx} V^i) - \mu(t, \bx) \cdot \Sigma\transpose \nabla_{\bx} V^i + \mu(t, \bx) \cdot \Sigma\transpose \nabla_{\bx} V^i + \half \text{Tr}(\Sigma\transpose \text{Hess}_{\bx} V^i \Sigma) = 0, \quad \forall i \in \mc{I},
\end{equation}
and take $\inf_{\alpha^i \in \MCA^i} H^i(t, \bx, \balpha, \Sigma\transpose \nabla_{\bx} V^i) - \mu(t, \bx) \cdot \Sigma\transpose \nabla_{\bx} V^i $ as the driver. Note that due to the coupling in \eqref{eq_V}, $\mu(t,\bx)$ needs to be identical across all $i \in \mc{I}$. Nevertheless, we think the choice in \eqref{def_BSDE} is the most natural one, without additional knowledge of $\miniH$.

In the next section, after we introduce the decoupling step, each sub-problem is also interpreted via a BSDE system (cf. \eqref{eq:FP_BSDE} and \eqref{eq:PU_BSDE}). There, though one can freely choose different $\mu^i$ for each backward process thanks to the decoupling, we still choose the forward process without a drift term as we did in \eqref{def_BSDE} for multiple reasons. We defer this explanation to Remark~\ref{rem_bsdechoice}. 

One final remark is that, with Assumption~\ref{assumption bsigma} one can apply a change of measure and make the controlled dynamics driftless as in \eqref{def_BSDE} under a different measure, which is indeed used in the proof of Theorem~\ref{thm_epsnash}.
\end{rem}

If solving directly, no matter which system (\eqref{def_HJB} or \eqref{def_BSDE}), one will encounter computational difficulties due to the high dimensionality of $\bm X_t$ or the large number of agents. To overcome this, we propose a two-step scheme in Section~\ref{sec_algorithm}, where we generalize the idea in \cite{HaHu:19} and offer two options for the first step. The convergence analysis with appropriate assumptions will be presented in Section~\ref{sec_analysis}.

\section{Algorithm}\label{sec_algorithm}
The two-step scheme for solving Markovian Nash equilibrium works as follows. We first decouple the problem \eqref{def_Xt}--\eqref{def_obj} into $N$ independent  sub-problems, for which we need to solve repeatedly and can solve in a parallel manner. Since each sub-problem may still be high-dimensional, we then solve them using deep neural networks with a reformulation in backward stochastic differential equations (BSDEs). Next, we describe the algorithms for each step in detail. 

\subsection{Step I: Decoupling}\label{sec_step1}
This step aims to decentralize the game, converting it into single-agent problems to be solved repeatedly. The algorithms start with an initial guess of the Nash equilibrium $\balpha^0 = [\alpha^{1,0}, \ldots, \alpha^{N,0}]$
and produce a sequence of strategies afterward, which we denote by $\balpha^1, \ldots \balpha^m, \ldots$.  
The following two options at this step differ in how the sequence is determined. Notationwise, $\balpha^m$ refers to the collection of all players' policies at stage $m$, and its $i^{th}$ component $\alpha^{i,m}$ refers to player $i$'s choice.

\begin{enumerate}
    \item \textbf{Fictitious Play.} 
    In this option of Step I, at each stage, each player faces an optimization problem \eqref{def_obj} while assuming that others are using their strategies from the previous stage as fixed strategies. In other words, at stage $m+1$, $\balpha^m$ is known to all players, and player $i$'s decision problem is
\begin{equation}\label{def_decoupled_obj}
\inf_{\alpha^i \in \mathbb{A}^i} J_0^i(\alpha^i; \balpha^{-i,m}),
\end{equation}
where $J_0^i$ is defined in \eqref{def_J}, and the state process $\bX_t$ follows \eqref{def_Xt} with $\balpha$ being replaced by $(\alpha^i; \balpha^{-i,m})$. Here $\balpha^{-i, m}$ represents the strategies of all players but player $i$ at stage $m$, and $(\alpha^i; \balpha^{-i,m})$ is a short notation of $(\alpha^{1, m}, \ldots, \alpha^{i-1,m}, \alpha^i, \alpha^{i+1, m}, \alpha^{N, m})$, which emphasis the parameter role of $\balpha^{-i, m}$. 

Under the Markovian framework, we denote by $V^{i, m+1}$ the problem value of player $i$ at stage $m$. Following the idea of fictitious play, it is the solution of the following HJB system
\begin{equation}\label{eq:FP_PDE}
\begin{dcases}
V_t^{i,m+1} + \inf_{\alpha^i \in \mc{A}^i} H^i(t, \bm x, (\alpha^i, \bm \alpha^{-i,m})(t, \bm x), \Sigma\transpose\nabla_{\bx} V^{i, m+1}) + \half \text{Tr}(\Sigma\transpose \text{Hess}_{\bx} V^{i, m+1} \Sigma) = 0,\\
V^{i, m+1}(T, \bx) = g^i(\bx).
\end{dcases}
\end{equation}

\item \textbf{Policy Update.} This is slightly different from fictitious play, where \emph{every} player follows her strategy from the previous stage to update the problem value. In this case, it is no longer an optimization, but a linear problem for the value function induced by the fix strategy $\balpha^m$:
\begin{equation}\label{eq:PU_PDE}
\begin{dcases}
V_t^{i,m+1} + H^i(t, \bm x,  \bm \alpha^{m}(t, \bm x),\Sigma\transpose \nabla_{\bx} V^{i, m+1}) + \half \text{Tr}(\Sigma\transpose \text{Hess}_{\bx} V^{i, m+1} \Sigma) = 0,\\
V^{i, m+1}(T, \bx) = g^i(\bx).
\end{dcases}
\end{equation}
\end{enumerate}

After solving out the decoupled PDE~\eqref{eq:FP_PDE} or \eqref{eq:PU_PDE}, at the end of stage $m+1$, a policy $\alpha^{i, m+1}$ is determined by
\begin{equation}\label{def_alphaast}
\alpha^{i, m+1}(t, \bm x) = \argmin_{\alpha^i \in \mc{A}^i} H^i(t, \bm x, (\alpha^i, \bm \alpha^{-i,m})(t, \bm x), \Sigma\transpose\nabla_{\bx} V^{i, m+1}(t, \bx)),
\end{equation}
and policies from all players together form $\balpha^{m+1}$. 

Note that for fictitious play algorithms, $\alpha^{i, m+1}$ is indeed the optimal strategy of problem \eqref{def_decoupled_obj}; while for policy update algorithms, the problem is linear, but we pretend that $V^{i, m+1}$ is the value of an optimization problem, and $\alpha^{i, m+1}$ is determined as if it is an optimizer. In short, the two algorithms differ at how $\balpha^{m+1}$ is update from $\balpha^m$. When interpreting via BSDEs, the different update rules result in slightly different drivers of the backward components, see equations \eqref{eq:FP_BSDE} and \eqref{eq:PU_BSDE} below. Nevertheless, the analysis based on the two algorithms presented in Theorems~\ref{thm_converge_revise} and \ref{thm_bsde_error_revise} follows similarly. 


\subsection{Step II: Solving Each Sub-problem via BSDE}\label{section_bsde_sovling}
For each sub-problem, described by \eqref{eq:FP_PDE} or \eqref{eq:PU_PDE}, we write down their BSDE counterpart:
\begin{align}\label{eq:FP_BSDE}
\begin{dcases}
 \bX_t = \bx_0 + \int_{0}^{t}\Sigma(s,\bX_s) \ud \bW_s,  \\
Y_t^{i, m+1} = g^i(\bX_T) + \int_{t}^{T}\hat H^i(s,\bX_s, \balpha^{-i, m}(s, \bX_s), Z_s^{i, m+1})\ud s - \int_{t}^{T}(Z_s^{i, m+1})\transpose \ud \bW_s, 
\end{dcases}
\end{align} 
where $\hat H^i$ is defined by
\begin{equation}\label{def_Hhat}
\hat H^i(t,\bx, \balpha^{-i}, p^i) = H^i(t, \bx, (a^i(t, \bx, \balpha^{-i}, p^i), \balpha^{-i}), p^i),    
\end{equation}
or
\begin{align}\label{eq:PU_BSDE}
\begin{dcases}
 \bX_t = \bx_0 + \int_{0}^{t}\Sigma(s,\bX_s) \ud \bW_s,  \\
Y_t^{i, m+1} = g^i(\bX_T) + \int_{t}^{T}H^i(s,\bX_s,\balpha^{m}(s, \bX_s), Z_s^{i, m+1})\ud s - \int_{t}^{T}(Z_s^{i, m+1})\transpose \ud \bW_s.
\end{dcases}
\end{align}
Here $\bx_0$ is a random variable whose range covers the states of interest.

\begin{rem}\label{rem_bsdechoice}
Note that, in both BSDEs above, we choose the forward process without a drift term for three reasons: (a) it avoids the involvement of $\bm{\alpha}^m$, and thus keeps the forward process the same from stage to stage; (b) the BSDEs can be vectorized (cf. \eqref{eq_FPvector_revise}) with a single forward process which coincides with forward component of \eqref{def_BSDE} (corresponding to the true solution), both will facilitate our analysis (c) numerically, this means only one forward process needs to be simulated for all $N$ sub-problems, which makes one iteration of step I--II more efficient. 

We also remark that both BSDEs are wellposed under Assumptions~\ref{assumption bsigma}--\ref{assumption 1 revise}, as the drivers $H^i$ and $\hat H^i$ are uniformly Lipschitz in $(\bx, p^i)$ and $g^i(\bX_T)$ is square integrable (cf. \cite[Theorem~4.3.1]{zhang2017backward}).
\end{rem}

For both sub-problems, the connection between the associated BSDEs and PDEs are the same:
\begin{equation}\label{eq_BSDEPDE}
    Y_t^{i, m+1} = V^{i, m+1}(t, \bX_t), \quad  Z_t^{i, m+1} = \Sigma(t, \bX_t)\transpose\nabla_{\bx} V^{i, m+1}(t, \bX_t),
\end{equation}
and according to~\eqref{def_alphaast}, both optimal policy processes at stage $m+1$ are expressed by
\begin{equation}
    \alpha^{i, m+1}_t = a^i(t, \bX_t, \balpha^{-i,m}_t, Z_t^{i, m+1}). 
\end{equation}
Therefore, it suffices to solve these two possibly high-dimensional BSDE systems by an efficient algorithm, which we shall describe and call deep BSDE in the sequel. To avoid repetition and cumbersome notation, the algorithms will be presented on a generic BSDE with possibly non-zero drift term:
\begin{equation}\label{eq:generic_BSDE}
\begin{dcases}
 X_t = x_0 + \int_0^t \mu(s, X_s) \ud s + \int_{0}^{t}\Sigma(s,X_s) \ud W_s,  \\
Y_t = g(X_T) + \int_{t}^{T}F(s,X_s, Z_s)\ud s - \int_{t}^{T}Z_s\transpose \ud W_s.
\end{dcases}
\end{equation}
The algorithm applied to the exact system \eqref{eq:FP_BSDE} and \eqref{eq:PU_BSDE} will be presented in Section~\ref{sec_step2}.


    
The deep BSDE is firstly introduced in \cite{EHaJe:17}, for solving high-dimensional parabolic PDEs. The idea is to solve a single variational form of \eqref{eq:generic_BSDE} after a temporal discretization version using deep neural networks. For a partition $\pi$ of size $N_T$ on the time interval $[0, T], 0 = t_0 < t_1 < \ldots < t_{N_T} = T$,  $\Delta t_k$ and $\Delta W_k$ are short notations for the time and Brownian motion increments respectively, and we denote by $\pinorm$ the mesh of this partition:
\begin{align}\label{def_numerics1}
\Delta t_k = t_{k+1} - t_k, \quad \Delta W_k = W_{ t_{k+1}}-W_{t_k}, \quad \|\pi\| = \max_{0 \leq k \leq N_T-1} \Delta t_k.
\end{align} 
We also define a step function $\pi(t)$, and a set $\mc{T}$ for later use:
\begin{equation}\label{def_numerics2}
   \pi(t) = t_k \text{ for } t \in [t_k, t_{k+1}), \quad  \mc{T}: = \{t_0, t_1, \ldots, t_{N_T-1}\}.
\end{equation}

The deep BSDE method solves the minimization problem:
\begin{align}
&\inf_{\psi_0\in \cN_0^{'},~\{\phi_k\in \cN_k\}_{k=0}^{N_T-1} } \EE|g(X_{T}^{\pi}) - Y_{T}^{\pi}|^2, \label{eq:disc_objective}\\
&s.t.~~ X_0^{\pi}=x_0, \quad Y_0^{\pi} =\psi_0(X_0^{\pi}), \quad Z_{t_k}^{\pi}=\phi_k(X_{t_k}^{\pi}), \quad k=0,\dots,N_T-1\notag \\
&\qquad X_{t_{k+1}}^{\pi} = X_{t_k}^{\pi} + \mu(t_k,X_{t_k}^{\pi})\Delta t_k +\Sigma(t_k,X_{t_k}^{\pi})\Delta W_{k}, \label{eq:disc_X_path} \\
 &\qquad Y_{t_{k+1}}^{\pi} = Y_{t_k}^{\pi} - F(t_k,X_{t_k}^{\pi},Z_{t_k}^{\pi})\Delta t_k + (Z_{t_k}^{\pi})\transpose\Delta W_k, \label{eq:disc_Y_path}
\end{align}
where 
$\cN_0^{'}$ and $\{\cN_k\}_{k=0}^{N_T-1}$ are hypothesis spaces related to deep neural networks, and for brevity, we use the notation $X_0^\pi$ for $X_{t_0}^\pi$, $X_T^\pi$ for $X_{t_{N_T}}^\pi$, {\it etc.} The goal is to find optimal deterministic maps $\psi_0^{\ast}, \{\phi_k^{\ast}\}_{k=0}^{N_T-1}$ such that the \emph{loss function} in \eqref{eq:disc_objective} is minimized. Intuitively, the smaller of \eqref{eq:disc_objective} provides the better approximation to the original problem \eqref{eq:generic_BSDE}. In practice, the expected value is replaced by the loss of a very deep neural network, which is formed by stacking all the subnetworks $\psi_0, \{\phi_k\}_{k=0}^{N_T-1}$ in sequence according to \eqref{eq:disc_Y_path}. The loss is computed by generating sample paths of $\{W_{t_k}\}_{k=0}^{N_T}$ and producing \eqref{eq:disc_X_path}--\eqref{eq:disc_Y_path}. At each stage, there are $N$ losses corresponding to $N$ sub-problems solved by the deep BSDE method.

Here we recall a convergence theorem for the deep BSDE method from \cite{HaLo:18}.
    
\begin{thm}\label{thm_deep_bsde}
For the generic BSDE \eqref{eq:generic_BSDE}, we assume:
\begin{enumerate}
    \item The functions $\mu$, $\Sigma$, $g$ and $F$ satisfy the following Lipschitz condition:
    \begin{align*}
        |\mu(t,x_1) - \mu(t,x_2)|^2 &+ \|\Sigma(t,x_1) - \Sigma(t,x_2)\|_F^2 + | F(t,x_1,p_1) - F(t,x_2,p_2)|^2 \\& + |g(x_1) - g(x_2)|^2\le L\left[|x_1 - x_2|^2 + |p_1 - p_2|^2\right],
    \end{align*}
    for a constant $L > 0$;
    \item The functions $\mu$, $\Sigma$ and $h$ are all 1/2-H\"older continuous with respect to $t$. For simplicity, we use $K$ for this H\"older constant.
    \item We  also use $K$ to denote the upper bound of $|\mu(0,0)|^2$, $\|\Sigma(0,0)\|_F^2$, $|F(0,0,0)|^2$ and $|g(0)|^2$.
\end{enumerate}
Then, we have the following two estimates:
\begin{align}
    &\sup_{t \in [0,T]} \EE|Y_t - Y_{\pi(t)}^\pi|^2 + \int_{0}^T \EE\|Z_t - Z_{\pi(t)}^\pi\|_F^2 \rmd t \le C\left[\pinorm + \EE|g(X_{T}^\pi) - Y_{T}^\pi|^2\right] \label{eq_deep_bsde_backward}\\
    &\inf_{\psi_0\in \cN_0^{'},~\{\phi_k\in \cN_k\}_{k=0}^{N_T-1} } \EE|g(X_{T}^{\pi}) - Y_{T}^{\pi}|^2 \le C\Bigg[\pinorm + \notag \\
    &\hspace{50pt} \left.\inf_{\psi_0\in \cN_0^{'},~\{\phi_k\in \cN_k\}_{k=0}^{N_T-1}} \left\{\EE|Y_0 - \psi_0(x_0)|^2 + \sum_{k = 0}^{N_T-1} \EE\|\hat{Z}_{t_k} - \phi_k(X_{t_k}^\pi)\|_F^2\Delta t_k\right\}\right] \label{eq_deep_bsde_forward},
\end{align}
where 
$\cN_0^{'}$ and $\{\cN_k\}_{k=0}^{N_T-1}$ are the hypothesis spaces for neural network architectures to approximate $Y_0^\pi$ and $Z_{t_k}^\pi$, 
$\pinorm$ and $\pi(t)$ are given in \eqref{def_numerics1}--\eqref{def_numerics2}, $\hat{Z}_{t_k} = (\Delta t_k)^{-1} \EE[\int_{t_k}^{t_{k+1}} Z_t \rmd t |X_{t_k}^\pi]$, and $C > 0$ is a constant only depending on $L$, $T$, $K$ and $\EE|\bx_0|^2$. 
\end{thm}

\begin{rem}\label{rem_thm1}
Theorem \ref{thm_deep_bsde} is a brief summary of Theorems 1 and 2 in \cite{HaLo:18}. The first inequality \eqref{eq_deep_bsde_backward} shows that the distance between the true solution of BSDE \eqref{eq:generic_BSDE} and the output of the deep BSDE method can be controlled by its loss function.  
In other words, in practice, the accuracy of the numerical solution is effectively indicated by the value of loss function.
The second inequality \eqref{eq_deep_bsde_forward} states that a small loss function of the deep BSDE method is attainable if the hypothesis spaces ($\cN_0'$ and $\{\cN_k\}_{k=0}^{N_T -1}$) can approximate specific functions well. Beyond Theorems 1 and 2 in \cite{HaLo:18}, there are still some theoretical issues remaining unresolved regarding the deep BSDE method, which are common in almost all the algorithms involving deep neural networks: First, it is unclear yet that what types of hypothesis spaces can approximate the specific functions in the deep BSDE method without the curse of dimensionality ({\it i.e.}, the number of parameters of neural networks grows at most polynomially both in dimension and the reciprocal of the approximation error). Second, even with suitable function spaces, it is hard to guarantee the 
optimization algorithm can find approximately the minimizer of the highly nonconvex loss function.
We refer the interested readers to \cite{EHaJe:17,HaJeE:18,HaLo:18} for more detailed descriptions and theoretical justifications of the deep BSDE method. Details on the implementation in this paper are presented in Section~\ref{sec_numerics}.
\end{rem}

\section{Convergence Analysis}\label{sec_analysis}

This section will provide the theoretical foundation for the deep fictitious play algorithm. Section~\ref{sec_analysis_step1} focuses on the decoupling step. Theorem~\ref{thm_converge_revise} proves the convergence to the true Nash equilibrium, if the decoupled sub-problems are solved exactly and repeatedly. Section~\ref{sec_step2} focuses on the numerical error on the deep BSDE algorithm for solving each sub-problem. Theorem~\ref{thm_bsde_error_revise} presents a game version of Theorem~\ref{thm_deep_bsde}. Section~\ref{sec_analysis_step3} combines the previous results, identifies the $\eps$-Nash equilibrium produced by deep fictitious play, and analyzes its numerical performance on the original game.

\subsection{Convergence Analysis of the Decoupling Step}\label{sec_analysis_step1}

In this section, we will focus on the convergence of the decoupling step, {\it i.e.}, how the systems defined by PDEs (\ref{eq:FP_PDE}) (fictitious play) or (\ref{eq:PU_PDE}) (policy update) converge to the system defined by PDEs (\ref{def_HJB}), or equivalently, how the corresponding BSDE systems (\ref{eq:FP_BSDE}) (fictitious play) or (\ref{eq:PU_BSDE}) (policy update) converge to the  BSDE system (\ref{def_BSDE}).

Throughout this paper, we shall use the following assumptions.

\begin{assu}\label{assumption 1 revise}
We shall use $|\cdot|$, $\|\cdot\|_F$ and $\|\cdot\|_S$ to denote the Euclidean norm, Frobenius norm and spectral norm, respectively. 
\begin{enumerate}[(1)]
\item The functions $\phi(t, \bx, \balpha): [0, T]\times \RR^n \times \MCA \to \RR^k $, $\Sigma(t, \bx): [0,T] \times \RR^n \to \RR^{n \times k}$, $\bm{f}(t, \bx, \balpha) = (f^1,f^2,\dots,f^N)\transpose(t,\bx, \balpha): [0,T] \times \RR^n \times \MCA \to \RR^N$ and $\bm{g}(\bx) = [g^1,g^2,\dots,g^N]\transpose(\bx): \RR^n \to \RR^N$ are Lipschitz with respect to $\bm{x}$ and $\bm{\alpha}$:
    \begin{align*}
        |\phi(t,\bm{x}_1,\bm{\alpha}_1) - \phi(t,\bm{x}_2,\bm{\alpha}_2)|^2 &\le L[ |\bm{x}_1 - \bm{x}_2|^2 + |\bm{\alpha}_1 - \bm{\alpha}_2|^2], \\
        \|\Sigma(t,\bm{x}_1) - \Sigma(t,\bm{x}_2)\|_F^2 &\le L|\bx_1-\bx_2|^2, \\
        |\bm{f}(t,\bm{x}_1,\bm{\alpha}_1) - \bm{f}(t,\bm{x}_2,\bm{\alpha}_2)|^2 &\le L[ |\bm{x}_1 - \bm{x}_2|^2 + |\bm{\alpha}_1 - \bm{\alpha}_2|^2], \\
        |\bm{g}(\bm{x}_1) - \bm{g}(\bm{x}_2)|^2 &\le L|\bx_1-\bx_2|^2.
    \end{align*}
    Here $L$ is a positive constant.
    \item The function $\bm{a}(t,\bm{x},\bm{\alpha},\bm{p})$ given in \eqref{def_a} is well-defined, and is Lipschitz  with respect to $\bm{x}$, $\balpha$ and $\bm{p}$:
    \begin{equation*}
        |\bm{a}(t,\bm{x}_1,\bm{\alpha}_1,\bm{p}_1) - \bm{a}(t,\bm{x}_2,\bm{\alpha}_2,\bm{p}_2)|^2 \le L(1-a_\alpha)[|\bm{x}_1 -\bm{x}_2|^2  + \|\bm{p}_1 - \bm{p}_2\|_F^2]   + a_\alpha|\bm{\alpha}_1 - \bm{\alpha}_2|^2,
    \end{equation*}
    with $0 < a_\alpha < 1$. Notice that this also implies that  $\bm{\alpha}(t,\bm{x},\bm{p})$ defined by \eqref{def_alphagameast} exists and is unique,  which is Lipschitz with respect to $\bm{x}$ and $\bm{p}$:
    \begin{equation*}
        |\bm{\alpha}(t,\bm{x}_1,\bm{p}_1)-\bm{\alpha}(t,\bm{x}_2,\bm{p}_2)|^2 \le L[|\bm{x}_1 - \bm{x}_2|^2 + \|\bm{p}_1 - \bm{p}_2\|_F^2].
    \end{equation*}
    \item\label{assumption 1-phisigma revise} The functions $\phi$ and $\Sigma$ are uniformly bounded:
    \begin{align*}
        &\|\Sigma(t,\bm{x})\|_S^2 \le M, \\
       &\max_{1\le i \le k} |\phi^i(t,\bm{x},\bm{\alpha})|^2 \le M.
    \end{align*}
    Here $\phi^i$ denotes the i-th component of $\phi$ and $M$ is a positive constant.
     \item The functions $\phi$, $\Sigma$, $\bm{f}$, $\bm{g}$ and $\bm{a}$ are all 1/2-H\"older continuous with respect to $t$. We shall use $K$ as the upper bound of all the H\"older constants.
    \item The constant $K$ is also the upper bound of constants $|\bm{a}(0,0,0,0)|^2$, $|\bm{f}(0,0,0)|^2$, $|\bm{g}(0)|^2$, $|\phi(0,0,0)|^2$ and $\|\Sigma(0,0)\|_F^2$.
\end{enumerate}
\end{assu}

\begin{assu}\label{assumption lipschitz solution}
There exists an adapted solution of the BSDE system (\ref{def_BSDE}) such that
\begin{equation}
    \EE \left[\sup_{0\le t \le T}(|\bm{X}_t|^2 +|\bm{Y}_t|^2) + \int_{0}^T \|\bm{Z}_t\|_F^2 \rmd t\right] < + \infty.
\end{equation}
Moreover, we assume that
\begin{equation*}
   \|\bm{Z}_t\|_S^2 \le M' \quad \mathbb{L}\times\mathbb{P}\text{-a.s.. } 
\end{equation*}
\end{assu}


We remark that Assumption~\ref{assumption 1 revise} is quite standard in the analysis of stochastic differential games and Assumption~\ref{assumption lipschitz solution} can be satisfied in several cases.
For instance, Assumption \ref{assumption lipschitz solution} holds true under Assumptions~\ref{assumption bsigma} and \ref{assumption 1 revise} with small time duration. We provide a detailed proof of this point (Proposition~\ref{prop:assump3_smalltime}) in the appendix. Note that the small time duration assumption is commonly seen in games, for instance in solving mean-field games \cite{CaDe2:17} and in the convergence of numerical schemes for mean-field games \cite{bayraktar2018numerical}.
Besides, through the nonlinear Feynman-Kac formula and the boundedness of $\Sigma$ in Assumption \ref{assumption 1 revise}, Assumption~\ref{assumption lipschitz solution} is also satisfied if the solution to the HJB system \eqref{def_HJB} is uniformly Lipschitz with respect to $\bx$. Specifically, with additional assumptions:
\begin{equation}\label{assump_HJB}
     \bm{f}, \bm{g} \text{ are bounded}, \text{ and } \Sigma\Sigma\transpose \text{ is uniformly nondegenerate}, 
\end{equation}
$V^i$ is indeed continuous and differentiable with bounded and continuous gradients on $[0,T) \times \RR^n$ (cf. \cite[Proposition~2.13]{CaDe1:17}). Therefore, using small time duration result (Proposition~\ref{prop:assump3_smalltime}) on $[T-\delta, T]$ for small $\delta$, one has the uniformly Lipschitz on $[0,T]$ and Assumption~\ref{assumption lipschitz solution} is fulfilled under \eqref{assump_HJB}. 
We also point out that Assumption \ref{assumption lipschitz solution} implies that the BSDE system \eqref{def_BSDE} has a unique adapted $L^2$-integrable solution, which will be shown in Proposition~\ref{prop:unique} in the appendix.

Recalling that $m$ is the index of the stage in the decoupling step, now we present the main result in this section regarding its convergence.
\begin{thm}\label{thm_converge_revise}
Under Assumptions \ref{assumption bsigma}, \ref{assumption 1 revise} and \ref{assumption lipschitz solution}, for any $\epsilon \in (0,1-a_\alpha)$, there exists a constant $C(\epsilon) > 0$  which only depends on $T$, $L$, $M$, $M'$ and $\epsilon$  such that
\begin{equation}\label{eq_converge_revise}
    \sup_{0 \le t \le T} \EE|\bm{Y}_t^{m} - \bm{Y}_t|^2 +\int_{0}^{T}\EE\|\bm{Z}_t^{m} - \bm{Z}_t\|_F^2 \rmd t + \int_{0}^{T}\EE|\bm{\alpha}_t^{m} - \bm{\alpha}^*_t|^2 \rmd t \le C(\epsilon) (a_\alpha + \epsilon)^m \int_{0}^{T}\EE|\bm{\alpha}_t^{0} - \bm{\alpha}^{*}_t|^2 \rmd t,
\end{equation}
where $(\bm{Y}^m_t,\bm{Z}^m_t)$ is defined by
\begin{equation*}
    \bm Y^m_t = [Y^{1,m}_t, \ldots, Y^{N,m}_t]\transpose, \quad \bm Z_t^m = [Z^{1,m}_t, \ldots, Z^{N,m}_t]
\end{equation*}
with $(Y^{i,m}_t, Z^{i,m}_t)$ from the BSDE systems (\ref{eq:FP_BSDE}) or (\ref{eq:PU_BSDE}), $(\bm{Y}_t, \bm{Z}_t)$ is defined in (\ref{def_BSDE}) , $\bm{\alpha}_t^{m} = \bm{\alpha}^m(t,\bm{X}_t)$ and $\bm{\alpha}^\ast_t = \bm{\alpha}^\ast(t,\bm{X}_t)$. 
\end{thm}

\begin{proof}
Theorem~\ref{thm_converge_revise} states the convergence of both fictitious play (according to~\eqref{eq:FP_BSDE}) and policy update (according to~\eqref{eq:PU_BSDE}).
The proofs of these two are very similar, and we shall focus on the fictitious play method for brevity.

To perform convergence analysis, we first rewrite the BSDE systems to show the explicit dependence on the players' strategies. 
For \eqref{def_BSDE}, it reads as
\begin{equation}\label{eq_BSDEvector_revise}
    \begin{dcases}
    \bm{X}_t = \bm{x}_0 + \int_{0}^{t}\Sigma(s,\bm{X}_s)\rmd \bm{W}_s, \\
    \bm{Y}_t = \bm{g}(\bm{X}_T) + \int_{t}^T  \bm{H}(s,\bm{X}_s,\bm{\alpha}^*_s,\bm{Z}_s)\rmd s - \int_{t}^T (\bm{Z}_s)\transpose \rmd \bm{W}_s, \\
    \bm{\alpha}^*_t =  \bm{a}(t,\bm{X}_t,\bm{\alpha}^*_t,\bm{Z}_t),
    \end{dcases}
\end{equation}
where $\bm{H}, \bm a$ are defined in \eqref{def_Hamiltonian} and \eqref{def_a}.
The rewritten system of \eqref{eq:FP_BSDE} is
\begin{equation}\label{eq_FPvector_revise}
    \begin{dcases}
    \bm{X}_t = \bm{x}_0 + \int_{0}^{t}\Sigma(s,\bm{X}_s)\rmd \bm{W}_s, \\
    \bm{Y}^{m+1}_t = \bm{g}(\bm{X}_T) + \int_{t}^T  \hat{\bm{H}}(s,\bm{X}_s,\bm{\alpha}_s^m,\bm{\alpha}_s^{m+1},\bm{Z}^{m+1}_s)\rmd s - \int_{t}^T (\bm{Z}_s^{m+1})\transpose \rmd \bm{W}_s, \\
    \bm{\alpha}_t^{m+1} =  \bm{a}(t,\bm{X}_t,\bm{\alpha}_t^m,\bm{Z}_t^{m+1}),
    \end{dcases}
\end{equation}
where $\hat {\bm{H}} = [\hat H^1, \ldots, \hat H^N]\transpose$, $\hat{H}^i(t,\bm{x},\bm{\xi},\bm{\gamma},\bm{p}) \equiv H^i(t,\bm{x},(\gamma^i,\bm{\xi}^{-i}),p^i)$ and $p^i$ stands for the $i^{th}$ column of $\bm p$.  Note that this is slightly an abuse of notation with \eqref{def_Hhat}, to show the driver's explicit dependence on $\balpha^{m+1}$.
Also note that the rewritten system \eqref{eq_FPvector_revise} is simply a condensed form of \eqref{eq:FP_BSDE}, concatenating all $Y_t^{i, m}$ into $\bY_t^m$, without changing its decoupled nature. This will also ease the notation in the following proof.

We now define $\delta \bm{H}^m_t = \hat{\bm{H}}(t,\bm{X}_t,\bm{\alpha}_t^m,\bm{\alpha}_t^{m+1},\bm{Z}^{m+1}_t) - \bm{H}(t,\bm{X}_t,\bm{\alpha}^*_t,\bm{Z}_t)$. 
Noticing that
\begin{align}
    \delta H_t^{i,m} &= \phi(t,\bm{X}_t,(\alpha_t^{i,m+1},\bm{\alpha}_t^{-i,m}))\cdot Z_t^{i,m+1} + f^i(t,\bm{X}_t,(\alpha_t^{i,m+1},\bm{\alpha}_t^{-i,m})) \notag \\
    &\quad - \phi(t,\bm{X}_t,\bm{\alpha}^*_t) \cdot Z_t^i - f^{i}(t,\bm{X}_t,\bm{\alpha}^*_t) \notag \\
    &= \phi(t,\bm{X}_t,(\alpha_t^{i,m+1},\bm{\alpha}_t^{-i,m})) \cdot (Z_t^{i,m+1} - Z_t^i) + [\phi(t,\bm{X}_t,(\alpha_t^{i,m+1},\bm{\alpha}_t^{-i,m})) - \phi(t,\bm{X}_t,\bm{\alpha}_t^{m})]\cdot Z_t^i\notag \\
    &\quad +[\phi(t,\bm{X}_t,\bm{\alpha}_t^{m}) - \phi(t,\bm{X}_t,\bm{\alpha}^*_t)]\cdot Z_t^i + [f^i(t,\bm{X}_t,(\alpha_t^{i,m+1},\bm{\alpha}_t^{-i,m})) - f^i(t,\bm{X}_t,\bm{\alpha}_t^{m})] \notag \\
    &\quad +[f^i(t,\bm{X}_t,\bm{\alpha}_t^{m}) - f^{i}(t,\bm{X}_t,\bm{\alpha}^*_t)].
\end{align}
Therefore, with Assumptions \ref{assumption 1 revise} and \ref{assumption lipschitz solution},
\begin{align}
    |\delta \bm{H}_t^m|^2 &\le C_1\left\{ \|\bm{Z}_t^{m+1} - \bm{Z}_t\|_F^2 + \sum_{i=1}^{N}(|Z_t^i|^2+1)|\alpha_t^{i,m+1} - \alpha_t^{i,m}|^2 + \|Z_t\|_S^2|\bm{\alpha}_t^m - \bm{\alpha}^*_t|^2 +  |\bm{\alpha}_t^m - \bm{\alpha}^*_t|^2\right\} \notag \\
    &\le C_2\left\{\|\bm{Z}_t^{m+1} - \bm{Z}_t\|_F^2 + |\bm{\alpha}_t^m - \bm{\alpha}^*_t|^2 + |\bm{\alpha}_t^{m+1}-\bm{\alpha}^*_t|^2\right\},
\end{align}
where $C_1,C_2$ are two positive constants only depending on $L$, $M$ and $M'$.

Next, we define $\delta \bm{Y}_t^{m} = \bm{Y}_t^m - \bm{Y}_t$, $\delta \bm{Z}_t^m = \bm{Z}_t^m - \bm{Z}_t$, $\delta \bm{\alpha}_t^m = \bm{\alpha}_t^m - \bm{\alpha}^*_t$. With equations \eqref{eq_BSDEvector_revise} and \eqref{eq_FPvector_revise}, we have
\begin{equation}
    \rmd \delta \bY_t^{m+1} = - \delta \bm{H}_t^m\rmd t + (\delta \bZ_t^{m+1})\transpose\rmd \bW_t.
\end{equation}
For any $\beta > 0$, by Ito's formula, taking expectation on both sides and integrating from $t$ to $T$, we have
\begin{align}
   e^{\beta t}\EE|\delta \bm{Y}^{m+1}_t|^2 + \int_{t}^T e^{\beta s}\EE\|\delta \bm{Z}^{m+1}_s\|_F^2\rmd s &= \int_{t}^Te^{\beta s}\EE[2\delta \bm{H}^m_s\cdot \delta \bm{Y}^{m+1}_s - \beta|\delta \bm{Y}^{m+1}_s|^2]\rmd s \\
   &\le \frac{1}{\beta} \int_{t}^Te^{\beta s}\EE|\delta \bm{H}^m_s|^2\rmd s,
\end{align}
where the inequality holds because $2\delta \bm{H}^m_s\cdot \delta \bY_s^{m+1} \le \beta^{-1}|\delta \bm{H}^m_s|^2 + \beta|\delta \bY_s^{m+1}|^2$. Then, taking the supremum with respect to $t$, we deduce
\begin{align}\label{converge_control}
    \sup_{0\le t \le T}\EE e^{\beta t}|\delta \bm{Y}_t^{m+1}|^2 + \int_{0}^Te^{\beta t}\EE\|\delta \bm{Z}_t^{m+1}\|_F^2\rmd t &\le \frac{1}{\beta}\int_{0}^{T}e^{\beta t}\EE|\delta \bm{H}_t^m|^2\rmd t \notag \\
    &\le \frac{C_2}{\beta}\int_{0}^{T}e^{\beta t}\EE[\|\delta\bm{Z}_t^{m+1}\|_F^2 +|\delta \bm{\alpha}_t^m|^2 +|\delta \bm{\alpha}_t^{m+1}|^2]\rmd t .
\end{align}
Choosing $\beta = C_2$, we can obtain
\begin{equation}\label{Y-control-revise}
    \sup_{0\le t \le T}\EE|\delta \bm{Y}_t^{m+1}|^2 \le e^{C_2T}\int_{0}^{T}[\EE|\delta \bm{\alpha}_t^m|^2 + \EE|\delta \bm{\alpha}_t^{m+1}|^2]\rmd t.
\end{equation}
For $\beta > C_2$, using inequality \eqref{converge_control} again, we have
\begin{equation}\label{Z-control-revise}
     \int_{0}^Te^{\beta t}\EE\|\delta \bm{Z}_t^{m+1}\|_F^2\rmd t \le \frac{C_2}{\beta -C_2}\int_{0}^{T}e^{\beta t}[\EE|\delta \bm{\alpha}_t^m|^2 + \EE|\delta \bm{\alpha}_t^{m+1}|^2]\rmd t.
\end{equation}
The Lipschitz condition of the function $\bm{a}$ and estimate \eqref{Z-control-revise} give that 
\begin{align}
    \int_{0}^{T}e^{\beta t}\EE|\delta \bm{\alpha}_t^{m+1}|^2\rmd t &\le \int_{0}^{T}e^{\beta t}[L(1-a_\alpha)\EE\|\delta \bm{Z}_t^{m+1}\|_F^2 + a_\alpha\EE|\delta \bm{\alpha}_t^{m}|^2]\rmd t  \notag\\
    &\le a_\alpha\int_{0}^{T}e^{\beta t}\EE|\delta \bm{\alpha}_t^{m}|^2\rmd t + \frac{LC_2}{\beta - C_2}\int_{0}^{T}e^{\beta t}[\EE|\delta \bm{\alpha}_t^m|^2 + \EE|\delta \bm{\alpha}_t^{m+1}|^2]\rmd t,
\end{align}
which is equivalent to (further assuming $\beta > (L+1) C_2)$
\begin{equation}\label{alpha_relation_revise}
    \int_{0}^{T}e^{\beta t}\EE|\delta \bm{\alpha}_t^{m+1}|^2 \rmd t \le \frac{\beta -C_2}{\beta - (L+1)C_2}\left(a_\alpha  + \frac{LC_2}{\beta - C_2}\right)\int_{0}^{T}e^{\beta t}\EE|\delta \bm{\alpha}_t^{m}|^2\rmd t. 
\end{equation}
For a given $\eps \in (0,  1- a_\alpha)$, we can choose $\beta $ large enough such that
\begin{equation}
    \frac{\beta -C_2}{\beta - (L+1)C_2}\left(a_\alpha  + \frac{LC_2}{\beta - C_2}\right) \le a_\alpha +\epsilon < 1.
\end{equation}
Then, there exists a constant $C(\epsilon) > 0$ that only depends on $T$, $L$, $M$, $M'$ and $\epsilon$  such that
\begin{equation}
    \int_{0}^{T}\EE|\delta \bm{\alpha}_t^m|^2\rmd t \le C(\epsilon)(a_\alpha + \epsilon)^m \int_{0}^{T}\EE|\bm{\alpha}_t^0 - \bm{\alpha}_t^{*}|^2 \rmd t.
\end{equation}
Combining the last inequality with inequalities \eqref{Y-control-revise} and \eqref{Z-control-revise}, we obtain our result.
\end{proof}

\begin{rem}
We remark that the convergence in Theorem~\ref{thm_converge_revise} holds for games with any size of $N$ instead of the focus in the numerical algorithm which is between 5 and 100, and is independent of any numerical scheme. 
\end{rem}

\subsection{Numerical Error Analysis}\label{sec_step2}

This section is dedicated to analyzing the numerical error introduced by the deep BSDE method when solving each sub-problem. Specifically, we aim to control the distance between $(\bm X_t, \bm Y_t, \bm Z_t)$ defined in \eqref{def_BSDE} and the discrete system $(\bX^\pi_{t_k}, \bY^{\pi, m}_{t_k}, \bZ^{\pi, m}_{t_k})$ satisfying:
\begin{align}\label{eq_Xdiscrete}
\begin{dcases}
\bm{X}_{t_{k+1}}^\pi = \bm{X}_{t_k}^\pi + \Sigma(t_k,\bm{X}_{t_k}^\pi)\Delta W_k, \quad \bm{X}_0^\pi = \bm{x}_0, \\
    \bm{Y}_{t_{k+1}}^{\pi,m+1} = \bm{Y}_{t_k}^{\pi,m+1} - \bm{h}^m(t_k,\bm{X}_{t_k}^\pi,\bm{Z}_{t_k}^{\pi,m+1})\Delta t_k + (\bm{Z}_{t_k}^{\pi,m+1})\transpose \Delta \bm{W}_k,
\end{dcases}
\end{align}
where $\bm{h}^m$ is either $[h^{1,m},\cdots, h^{N,m}]\transpose$ with $h^{i,m}(t,\bm{x},\bm{p}) = \inf_{\alpha^i \in \mathcal{A}^i}H^i(t,\bm{x},(\alpha^i,\bm{\alpha}^{-i,\pi,m}(t,\bm{x})),\bm{p})$ when decoupled through fictitious play, or $\bm{h}^m(t,\bm{x},\bm{p}) = \bm{H}(t,\bm{x},\bm{\alpha}^{\pi,m}(t,\bm{x}),\bm{p})$ when decoupled through policy update. As stated in Section~\ref{section_bsde_sovling}, $\bY_0^{\pi, m+1}$ and $\bZ_{t_k}^{\pi, m+1}$ are paramterized by neural networks,
\begin{align*}
   \bm{Y}_0^{\pi,m+1} = \bm{\psi}_0^{m+1}(\bm{X}_0^\pi), \quad \bm{Z}_{t_k}^{\pi,m+1} = \bm{\phi}_k^{m+1}(\bm{X}_{t_k}^\pi), 
\end{align*}
where $\bm \psi_0^{m}$ and $\{\bm \phi_k^{m}\}_{k=0}^{N_T-1}$ are the optimal deterministic maps determined at stage $m$ that belongs to the hypothesis spaces (cf. Section~\ref{section_bsde_sovling}).
Afterwards, the $(m+1)^{th}$-stage policies defined on $\mathcal{T} \times \bR^n$ are updated by:
\begin{equation}
    \bm{\alpha}^{\pi,m+1}(t,\bm{x}) = \bm{a}(t,\bm{x},\bm{\alpha}^{\pi,m}(t,\bm{x}),\bm{\phi}^{m+1}(t,\bm{x})), \quad \forall (t,\bm{x}) \in \mathcal{T}\times \bR^n \label{bsde_discrete}
\end{equation}
where $\bm\phi^{m}(t_k,\bm{x}) = \bm \phi_k^m(\bm{x})$. 
Note that the above notation is simply a vector form of the deep BSDE method applied to system \eqref{eq:FP_BSDE} or \eqref{eq:PU_BSDE}. It does not change the decoupling nature of the deep fictitious play algorithm, i.e., each entry $(Y^{i, \pi, m}, Z^{i, \pi, m})$ in $(\bY^{\pi, m}, \bZ^{\pi, m})$ still solves its own problem.

Initially, we hope to apply Theorem \ref{thm_deep_bsde} to the BSDE system $\eqref{eq:FP_BSDE}$ and $\eqref{eq:PU_BSDE}$. By the game feature and the decoupling scheme, stage $m+1$'s estimates rely on the regularity of stage $m$'s policy $\balpha^m(t,\bx)$ (see definition in \eqref{def_alphaast}). Specifically, it requires the following condition, in addition to Assumption \ref{assumption 1 revise}:
\begin{equation} \label{regularity_alpha_m}
    |\bm{\alpha}^m(t_1,\bm{x}_1) - \bm{\alpha}^m(t_2,\bm{x}_2)|^2 \le L[|t_1 - t_2| + |\bm{x}_1 - \bm{x}_2|^2]. 
\end{equation}
However, in general, this property is not inherited from stage to stage. To circumvent this issue, we introduce a projection operator, which needs to be applied at the end of each stage. Along this line, we need the following assumption.

\begin{assu}\label{assumption gradientv}
The optimal policy $\bm{\alpha}^{*}$ as a function on $[0,T]\times \bR^n$ is Lipschitz with respect to $\bx$ and 1/2-H\"older continuous with respect to $t$, i.e.,
\begin{equation}
    |\bm{\alpha}^{*}(t_1,\bm{x}_1) - \bm{\alpha}^{*}(t_2,\bm{x}_2)|^2 \le L(|t_1-t_2| + |\bx_1 - \bx_2|^2).
\end{equation}
We also assume that $|\bm{\alpha}^\ast(t,\bm{x})|^2 \le L(1 + |\bm{x}|^2)$ for any $(t,\bm{x}) \in [0,T]\times\bR^n$. 
\end{assu}


Recall the set $\mc{T}$ containing all endpoints of the partition $\pi$ on $[0,T]$ from \eqref{def_numerics2}, for any $\eta \ge 0$ we define a Hilbert space on $\mc{T} \times \RR^n$:
\begin{equation}\label{def_hilbert_space}
    \mathcal{H}_\eta^\pi = \left\{\bm{\alpha}:\text{ measurable functions from }\mc{T}\times \bR^n \text{ to } \bR^{Nd_{\bm{\alpha}}}, \sum_{k=0}^{N_T-1}e^{\eta t_k}\EE|\bm{\alpha}(t_k,\bm{X}_{t_k}^\pi)|^2\Delta t_k < +\infty\right\},
\end{equation}
with norm $\|\bm{\alpha}\|_{\mathcal{H}_\eta^\pi}^2 := \sum_{k=0}^{N_T-1}e^{\eta t_k}\EE|\bm{\alpha}(t_k,\bm{X}_{t_k}^\pi)|^2\Delta t_k$,
and a subset 
\begin{equation*}
    \mathcal{N}^\pi = \left\{\bm{\alpha}:\mc{T}\times \bR^n \hookrightarrow \bR^{Nd_{\bm{\alpha}}}, |\bm{\alpha}(t_1,\bm{x}_1) - \bm{\alpha}(t_2,\bm{x}_2)|^2 \le L'[|t_1-t_2| + |\bx_1 - \bx_2|^2], |\bm{\alpha}(t,\bx)|^2 \le L'(1 + |\bx|^2)\right\}\footnote{Note that here $t_1$ and $t_2$ refer to two arbitrary points in $\mc{T}$, but not the first and second endpoints in the partition $\pi$. Later, depending on the contexts, the notation $t_1$ and $t_2$ in $|t_1 - t_2|$ may also refer to arbitrary points in $[0,T]$.}
\end{equation*}
with a constant $L' \ge L$. We claim that $\mathcal{N}^\pi$ is a closed convex subset of $\mathcal{H}_\eta^\pi$, so the projection $\mathrm{P}_{\mathcal{N}^\pi,\eta}$ from $\mathcal{H}_\eta^\pi$ to $\mathcal{N}^\pi$ exists and does not increase distance (cf., \cite[Chapter~5]{brezis2010functional}):
\begin{equation}\label{projection_contraction}
    \|\mathrm{P}_{\mathcal{N}^\pi,\eta} (f_1) - \mathrm{P}_{\mathcal{N}^\pi,\eta} (f_2)\|_{\mathcal{H}_\eta^\pi} \le \|f_1 - f_2\|_{\mathcal{H}_\eta^\pi} \quad \forall f_1, f_2 \in \mathcal{H}_\eta^\pi.
\end{equation}
The convexity of $\mathcal{N}^\pi$ is straightforward. To see the closedness, let $\{\bm{\alpha}_j\}_{j \ge 1}$ be a convergent sequence in $\mathcal{N}^\pi$, then there exists a subsequence of $\{\bm{\alpha}_j\}_{j \ge 1}$, denoted by $\{\bm{\alpha}_{j_k}\}_{k \ge 1}$, that converges to $\bm{\alpha}_\infty$ a.s.. Since
\begin{equation*}
    |\bm{\alpha}_{j_k}(t_1,\bm{x}_1) - \bm{\alpha}_{j_k}(t_2,\bm{x}_2)|^2 \le L'[|t_1-t_2| + |\bx_1 - \bx_2|^2], \quad |\bm{\alpha}_{j_k}(t,\bx)|^2 \le L'(1 + |\bx|^2),
\end{equation*}
and let $k \rightarrow +\infty$, we obtain
\begin{equation*}
    |\bm{\alpha}_{\infty}(t_1,\bm{x}_1) - \bm{\alpha}_{\infty}(t_2,\bm{x}_2)|^2 \le L'[|t_1-t_2| + |\bx_1 - \bx_2|^2], \quad |\bm{\alpha}_{\infty}(t,\bx)|^2 \le L'(1 + |\bx|^2), a.s..
\end{equation*}
Noticing that the functions in $\mathcal{N}^\pi$ and $\mathcal{H}_\eta^\pi$ are identical if they agree almost everywhere, we can conclude $\bm{\alpha}_{\infty} \in \mathcal{N}^\pi$ and $\mathcal{N}^\pi$ is closed.


Therefore, we are able to apply the projection operator $\mathrm{P}_{\cN^\pi, \eta}$ at the end of each stage, i.e., we change equation \eqref{bsde_discrete} to
\begin{align}
&\tilde{\bm{\alpha}}^{\pi,m+1}(t,\bm{x}) = \bm{a}(t,\bm{x},\bm{\alpha}^{\pi,m}(t,\bm{x}),\bm{\phi}^{m+1}(t,\bm{x})),
\label{eq:def_proj1}
\\
&\bm{\alpha}^{\pi,m+1} = \mathrm{P}_{\mathcal{N}^\pi,\eta}(\tilde{\bm{\alpha}}^{\pi,m+1}).
\end{align}
By this definition, the numerical solution ${\bm{\alpha}}^{\pi,m+1}(t,\bm{x})$ in fact implicitly depends on the value of $\eta$. We suppress this dependence for brevity of notation.
The main theorem in this section is as follows. 

\begin{thm}\label{thm_bsde_error_revise}
Under Assumptions \ref{assumption bsigma}--\ref{assumption gradientv}, let $\bm{\alpha}^{\pi,0}: \mc{T} \times \bR^n \mapsto \mathcal{A}$ be a measurable function satisfying:
\begin{equation}\label{alpha_0_regularity_revise}
    |\bm{\alpha}^{\pi,0}(t_1,\bm{x}_1) - \bm{\alpha}^{\pi,0}(t_2,\bm{x}_2)|^2 \le L'[|t_1- t_2| + |\bm{x}_1 - \bm{x}_2|^2], \quad |\bm{\alpha}^{\pi,0}(t,\bm
    {x})|^2 \le L'(1+|\bm{x}|^2).
\end{equation}
Then, for any $\epsilon \in (0,1-a_\alpha)$, assuming that $\eta> \eta_\epsilon$ in \eqref{def_hilbert_space}, where $\eta_\epsilon$ is a constant depending on $T$, $L$, $M$, $M'$ and $\epsilon$, we have
\begin{align}\label{eq_deep_bsde_error_revise}
    \sup_{t\in[0,T]} \EE&|\bm{Y}_t - \bm{Y}^{\pi,m}_{\pi(t)}|^2 + \int_{0}^{T} \EE\|\bm{Z}_t - \bm{Z}^{\pi,m}_{\pi(t)}\|^2_F\,\rmd t + \int_{0}^{T} \EE|\bm{\alpha}^*_t - \bm{\alpha}^{\pi,m}_{\pi(t)}|^2\,\rmd t  \notag \\
    & \le C(\eta,\epsilon)\left[\pinorm+  (a_\alpha + \epsilon)^m\int_{0}^{T}\EE\left|\bm{\alpha}^{*}_t - \bm{\alpha}_{\pi(t)}^{\pi,0}\right|^2 \rmd t + \sum_{j = 1}^{m} (a_\alpha + \epsilon)^{m-j} \EE\left|\bm{g}(\bm{X}_{T}^\pi) - \bm{Y}_{T}^{\pi,j}\right|^2\right],
\end{align}
where $(\bX^\pi_{t_k}, \bY^{\pi, m}_{t_k}, \bZ^{\pi, m}_{t_k})$ is defined in \eqref{eq_Xdiscrete}, $\balpha_{\pi(t)}^{\pi, m}  \equiv \bm{\alpha}_{t_k}^{\pi,m} = \bm{\alpha}^{\pi,m}(t_k,\bm{X}_{t_k}^\pi)$ for $t \in [t_k, t_{k+1})$, and $C(\eta,\epsilon) > 0$ is a constant depending only on $T$, $L$, $M$, $M'$, $K$, $L'$, $\EE|\bx_0|^2$, $\eta$ and $\epsilon$. Here $(\bX^\pi_{t_k}, \bY^{\pi, m}_{t_k}, \bZ^{\pi, m}_{t_k})$ represents either the discrete BSDE system using fictitious play or policy update in the decoupling step, depending on the definition of $\bm h$ in \eqref{eq_Xdiscrete}.

Next, with a slight abuse of notation (see Remark~\ref{rem_abuseofnotation} (2) for details), we define
$(\bm{Y}_t^m, \bm{Z}_t^m)$ as
\begin{equation*}
    \bY_t^m = [Y_t^{1, m}, \ldots, Y_t^{N, m}]\transpose, \quad \bZ_t^m = [Z_t^{1.m}, \ldots, Z_t^{N,m}]
\end{equation*}
with $(Y_t^{i,m}, Z_t^{i,m})$ from the BSDE systems \eqref{eq:FP_BSDE} in the setting of fictitious play or \eqref{eq:PU_BSDE} in the setting of policy update, in which the previous stage policy is given by the extension of the numerical approximation in time
\begin{equation}\label{alpha_interpolation_revise}
    \bm{\alpha}^m(t,\bm{x}) = \inf_{t'\in \mathcal{T}}[\bm{\alpha}^{\pi,m}(t',\bm{x}) + L'|t'-t|^{\frac{1}{2}}].
\end{equation}
Then we have the following inequality
\begin{align}\label{eq_loss_bound_revise}
    &\quad\,\inf_{\bm{\psi_0}^m \in \mathcal{N}_0', \{\bm{\phi}_k^m \in \mathcal{N}_k\}_{k=0}^{N_T-1}}\EE|\bm{g}(\bm{X}_{T}^\pi) - \bm{Y}_{T}^{\pi,m}|^2 \notag \\
    &\le C\left[\pinorm+  \inf_{\bm{\psi}_0^m \in \mathcal{N}'_0, \{\bm{\phi}_k^m \in \mathcal{N}_k\}_{k=0}^{N_T-1}}\left\{\EE|\bm{Y}_0^{m} - \bm{\psi}_0^m(\bm{x}_0)|^2 + \sum_{k = 0}^{N_T-1} \EE\|\hat{\bm{Z}}_{t_k}^{m} - \bm{\phi}_k^m(\bm{X}_{t_k}^\pi)\|_F^2\Delta t_k \right\}\right],
\end{align}
where $\cN_0^{'}$ and $\{\cN_k\}_{k=0}^{N_T-1}$ are hypothesis spaces for neural network architectures to approximate $\bY_0$ and $\bZ_{t_k}$,
$\hat{\bm{Z}}_{t_k}^m = (\Delta t_k)^{-1}\EE[\int_{t_k}^{t_{k+1}}\bm{Z}_t^m \rmd t | \bm{X}_{t_k}^\pi]$ and $C$ is a constant only depending on $T$, $L$, $M$, $M'$, $K$, $L'$ and $\EE|\bx_0|^2$.
We still refer $\cN_0'$ and $\cN_k$ as the hypothesis spaces for $\bm\psi_0^m: \RR^n \to \RR^N$, $\bm \phi_k^m:  \RR^n \to \RR^{k \times N}$, without introducing superscript $m$ to indicate the stage.
\end{thm}

\begin{rem}\label{rem_abuseofnotation}
We have the following remarks regarding Theorem~\ref{thm_bsde_error_revise}:
\begin{enumerate}[(1)]
    \item The interpretation of Theorem~\ref{thm_bsde_error_revise} is similar to that of Theorem~\ref{thm_deep_bsde}. The first inequality \eqref{eq_deep_bsde_error_revise} shows that the distance between the true solution of BSDE \eqref{def_BSDE} and the output of the deep BSDE method at stage $m$ can be controlled together by the mesh size, the error of the initial policy and the loss functions achieved at all the previous stages. The second inequality \eqref{eq_loss_bound_revise} states that the loss function of deep BSDE method at each stage is small if the approximation capability of
    the parametric function spaces ( ($\cN_0'$ and $\{\cN_k\}_{k=0}^{N_T -1}$) is high. The overall message conveyed in Theorem~\ref{thm_bsde_error_revise} is that, if the deep BSDE method can solve each sub-problem accurately enough, the deep fictitious play method will produce a strategy close to the Nash equilibrium.
    
    \item Note that there is a slight abuse of notation in the statement of Theorem~\ref{thm_bsde_error_revise}, since $(\bm Y_t^m, \bm Z_t^m)$ and $\balpha^m$ have already been introduced in Sections~\ref{sec_analysis_step1} and \ref{sec_step1}, as the theoretical solution from the decoupling step at stage $m$. In this section, to avoid  introducing further complicated notations, we still refer $(\bm Y_t^m, \bm Z_t^m)$ as the theoretical solution depending on $\balpha^{m-1}$, but $\balpha^{m-1}$ is the interpolation \eqref{alpha_interpolation_revise} of the deep BSDE solution $\balpha^{\pi, m-1}$ at stage $m-1$.  Nevertheless, the relation between $(\bY^m, \bZ^m)$ and the interpolated strategy $\balpha^{m-1}$ in theorem 3 remains the same as the relation between $(\bY^m, \bZ^m)$ and the exact strategy $\balpha^{m-1}$ in Theorem 2, thus some estimates follow using the same derivations as in the proof of Theorem \ref{thm_converge_revise}. In particular, we can obtain that there exists  positive constants $\beta_0$ and $C$ only depending on $T$, $L$, $M$ and $M'$ such that for any $\beta > \beta_0$,
    \begin{align}
        \sup_{0\le t \le T}\EE|\bY_t - \bY_t^{m}|^2 &\le C\int_{0}^{T}\EE|\balpha_t^* - \balpha^{m-1}(t,\bX_t)|^2\rmd t, \label{Y_control_thm4_revise}\\
        \int_{0}^{T}e^{\beta t}\EE\|\bZ_t - \bZ_t^{m}\|_F^2 &\le \frac{C}{\beta - \beta_0}\int_{0}^{T}e^{\beta t}\EE|\balpha_t^* - \balpha^{m-1}(t,\bX_t)|^2\rmd t \label{Z_control_thm4_revise}.
    \end{align}
    where $\balpha^{m-1}$ follows \eqref{alpha_interpolation_revise} and is the interpolation of strategies computed numerically at stage $m-1$. 
\end{enumerate}
\end{rem}

\begin{proof}
 Throughout this proof, we will use $C$ to denote a positive constant depending only on $T$, $L$, $M$, $M'$, $K$, $L'$ and $\EE|\bx_0|^2$ and use $C(\cdot)$ 
 to denote a positive constant depending on all the above constants and the arguments represented by $\cdot$. Both $C$ and $C(\cdot)$ may vary from line to line.
 
 Since $\bm{\alpha}^{\pi,m} \in \mathcal{N}^\pi$, with conditions \eqref{alpha_0_regularity_revise} and \eqref{alpha_interpolation_revise}, we obtain (cf. \cite{mcshane1934extension})
 \begin{equation}\label{alpha_m_regularity_revise}
      |\bm{\alpha}^m(t_1,\bm{x}_1) - \bm{\alpha}^m(t_2,\bm{x}_2)|^2 \le L'[|t_1 - t_2| + |\bm{x}_1 - \bm{x}_2|^2], \quad |\bm{\alpha}^m(t,\bx)| \le L'(1 + \sqrt{\|\pi\|} + |\bx|^2).
 \end{equation}
 Thus, the inequality \eqref{eq_loss_bound_revise} follows from Theorem \ref{thm_deep_bsde}.
 
 We next prove the inequality \eqref{eq_deep_bsde_error_revise}. As before, we will focus on proving the case of fictitious play in the sequel, and we claim that the statements also hold for policy update using a similar argument.
 
 Recalling the $\{\bX_{t_k}^\pi\}_{0\le k \le N_{T}-1}$ in \eqref{eq_Xdiscrete}, we then define the Euler-type scheme for BSDE system \eqref{def_BSDE} as follow:
 \begin{equation}
    \begin{dcases}
    \bY^\pi_{t_k} = \EE[\bY^\pi_{t_{k+1}}|X_{t_k}] + \miniH(t_k,\bX^\pi_{t_k},\bZ^\pi_{t_k})\Delta t_k, \quad \bY^\pi_{T} = g(\bX^\pi_{T}), \\
    \bZ^\pi_{t_k} = \frac{1}{\Delta t_k}\EE[(\bY^\pi_{t_{k+1}})\transpose\Delta \bW_k|\bX_{t_k}^\pi], \quad \forall k = 0,1, \dots,N_T-1.
    \end{dcases}
    \end{equation}
    With Assumptions \ref{assumption bsigma} and \ref{assumption 1 revise}, classical estimations of the discretization error gives
    \begin{equation}\label{X_discrete_error_revise}
        \sup_{0\le t\le T} [\EE|\bX_t - \bX_{\pi(t)}^\pi|^2 + \EE|\bY_t - \bY_{\pi(t)}^\pi|^2] + \int_{0}^{T}\EE\|\bZ_t - \bZ_{\pi(t)}^\pi\|_F^2 \rmd t \le C\|\pi\|.
    \end{equation}
    
For the $\bZ$-part error, we decompose it into two terms by the Cauchy-Schwartz inequality:
\begin{equation}\label{thm4_eq1_revise}
    \int_0^T\EE\|\bZ_t - \bZ_{\pi(t)}^{\pi,m+1}\|_F^2  \ud t \leq 2\int_0^T \EE\|\bZ_t - \bZ_t^{m+1}\|_F^2 \ud t  + 2 \int_0^T \EE\|\bZ_t^{m+1} - \bZ_{\pi(t)}^{\pi,m+1}\|_F^2 \ud t.
\end{equation}
A similar inequality can be written on the $\bY$-part error. For both of them, the second term is taken care by applying Theorem~\ref{thm_deep_bsde} to $(\bY_t^m, \bZ_t^m)$. More precisely, under Assumptions \ref{assumption 1 revise}--\ref{assumption gradientv} and \eqref{alpha_m_regularity_revise}, one has:
\begin{align}\label{deep_bsde_error_revise}
   \sup_{t\in [0,T]} \EE|\bm{Y}^{m+1}_t - \bm{Y}_{\pi(t)}^{\pi,m+1}|^2 + \int_{0}^T\EE\|\bm{Z}^{m+1}_t -\bm{Z}^{\pi,m+1}_{\pi(t)}\|_F^2 \rmd t \le C\left[\pinorm+ \EE|\bm{Y}_{T}^{\pi,m+1} - \bm{g}(\bm{X}_{T}^\pi)|^2\right],
\end{align}
where $(\bY_{\pi(t)}^{\pi, m}, \bZ_{\pi(t)}^{\pi, m})$ is defined in \eqref{eq_Xdiscrete}. For the first term in \eqref{thm4_eq1_revise}, we recall the inequality \eqref{Z_control_thm4_revise}  (choosing $\beta = \beta_0+1$) and deduce
\begin{align}\label{thm4_eq2_revise}
    \int_{0}^{T}\EE\|\bZ_t - \bZ_t^{m+1}\|_F^2\rmd t &\le C\int_{0}^{T}\EE|\balpha_t^{*} - \balpha^m(t,\bX_t)|^2\rmd t \notag\\
    &\le C[\int_{0}^{T}\EE|\balpha_t^{*} - \balpha_{\pi(t)}^{\pi,m}|^2\rmd t + \|\pi\|], 
\end{align}
where we have used 
\begin{equation}\label{thm4_eq3_revise}
\int_{0}^T\EE|\bm{\alpha}_{\pi(t)}^{\pi,m} - \bm{\alpha}^m(t,\bm{X}_t)|^2 \rmd t \le C\pinorm
\end{equation}
as a consequence of  \eqref{alpha_m_regularity_revise} and \eqref{X_discrete_error_revise}. Combining \eqref{thm4_eq1_revise}--\eqref{thm4_eq2_revise},  we claim that
\begin{equation}\label{thm4_eq4_revise}
    \int_{0}^T\EE\|\bZ_t - \bZ_{\pi(t)}^{\pi,m+1}\|_F^2\rmd t \le C[\int_{0}^T\EE|\balpha_t^{*} - \balpha_{\pi(t)}^{\pi,m}|^2\rmd t + \|\pi\| + \EE|\bY_T^{\pi,m+1} - \bm{g}(\bX_T^\pi)|^2].
\end{equation}
 Using equations \eqref{Y_control_thm4_revise}, \eqref{deep_bsde_error_revise} and \eqref{thm4_eq3_revise}, we can similarly obtain that
 \begin{equation}\label{thm4_eq4'_revise}
    \sup_{0\le t \le T}\EE|\bY_t - \bY_{\pi(t)}^{\pi,m+1}| \le C[\int_{0}^T\EE|\balpha_t^{*} - \balpha_{\pi(t)}^{\pi,m}|^2\rmd t + \|\pi\| + \EE|\bY_T^{\pi,m+1} - \bm{g}(\bX_T^\pi)|^2].
 \end{equation}
 
 We next derive an estimate that is useful in controlling the $\balpha$-part error. We first require $\eta_{\epsilon} > \beta_0$, then we have
\begin{align}\label{thm4_eq5_revise}
&\quad\sum_{k = 0}^{N_T - 1}e^{\eta t_k}\EE\|\bm{Z}_{t_k}^\pi - \bm{Z}_{t_k}^{\pi,m+1}\|_F^2 \Delta t_k \notag = \int_{0}^{T}e^{\eta \pi(t)}\EE\|\bm{Z}_{\pi(t)}^\pi - \bm{Z}_{\pi(t)}^{\pi,m+1}\|_F^2 \rmd t \notag \\
&\le 3 \int_{0}^{T}e^{\eta t}[\EE\|\bm{Z}_t - \bm{Z}_t^{m+1}\|_F^2+ \EE\|\bm{Z}_t - \bm{Z}_{\pi(t)}^\pi\|_F^2 + \EE\|\bm{Z}_t^{m+1} - \bm{Z}_{\pi(t)}^{\pi,m+1}\|_F^2] \rmd t \notag \\&\le3\int_{0}^{T}e^{\eta t}\EE\|\bm{Z}_t - \bm{Z}_t^{m+1}\|_F^2\rmd t + C(\eta)\left[\pinorm +\EE|\bm{Y}_{T}^{\pi,m+1} - \bm{g}(\bm{X}_{T}^\pi)|^2 \right] \notag\\
 &\le \frac{C}{\eta - \beta_0}\int_{0}^{T}e^{\eta t}\EE|\bm{\alpha}_t^{*} - \bm{\alpha}_{\pi(t)}^{\pi,m}|^2 \rmd t 
  + C(\eta) \left[\pinorm + \EE|\bm{Y}_{T}^{\pi,m+1} - \bm{g}(\bm{X}_{T}^\pi)|^2 \right],
\end{align}
where we have used the Cauchy-Schwart inequality, inequalities \eqref{Z_control_thm4_revise}, \eqref{X_discrete_error_revise}, \eqref{deep_bsde_error_revise} and \eqref{thm4_eq3_revise}. 

For the $\balpha$-part error, it suffices to control $\int_{0}^{T}e^{\eta t}\EE|\bm{\alpha}_t^{*} - \bm{\alpha}_{\pi(t)}^{\pi,m}|^2 \rmd t$ and we plan to 
\begin{enumerate}[(1)]
    \item\label{alpha_express_revise} express $\text{I} :=  \sum_{k=0}^{N_T - 1}e^{\eta t_k}\EE|\bm{\alpha}^{*} - \tilde{\bm{\alpha}}^{\pi,m+1}|^2(t_k,\bm{X}_{t_k}^\pi)\Delta t_k$ in terms of $\int_{0}^{T}e^{\eta t}\EE|\bm{\alpha}_t^{*} - \bm{\alpha}_{\pi(t)}^{\pi,m}|^2 \rmd t$;
    \item\label{alpha_projection_revise} obtain the estimate of $\text{II} :=  \sum_{k=0}^{N_T - 1}e^{\eta t_k}\EE|\bm{\alpha}^{*} - {\bm{\alpha}}^{\pi,m+1}|^2(t_k,\bm{X}_{t_k}^\pi)\Delta t_k \leq \text{I}$ by the property \eqref{projection_contraction} of $P_{\cN^\pi}$;
    \item\label{alpha_time_revise} take care the difference between the  $\balpha$-part error and II by III which is defined by:
    \begin{equation}\label{thm4_eq6_revise}
        \text{III} := \int_{0}^{T}e^{\eta t}\EE|\bm{\alpha}_t^{*} - \bm{\alpha}^{*}(\pi(t),\bm{X}_{\pi(t)}^\pi)|^2\rmd t \le C(\eta)\pinorm.
    \end{equation}
\end{enumerate}
Step \eqref{alpha_projection_revise} follows from the fact that $\balpha^{\pi, m+1}$ is defined as the projection of $\tilde \balpha^{\pi, m+1}$ into $\cN^\pi$, and that $\bm{\alpha}^{*} \in \mathcal{N}^\pi$ if viewed as a function on $\mc{T}\times \bR^n$. Step~\eqref{alpha_time_revise} is a consequence of Assumption~\ref{assumption gradientv} and \eqref{X_discrete_error_revise}. So it remains to address step~\eqref{alpha_express_revise}.
 
 To this end, we define $\bm{\alpha}_{t_k}^{\pi,*} = \bm{\alpha}(t_k,\bm{X}_{t_k}^\pi,\bm{Z}_{t_k}^\pi)$, then $\bm{\alpha}_{t_k}^{\pi,*} = \bm{a}(t_k,\bm{X}_{t_k}^{\pi},\bm{\alpha}_{t_k}^{\pi, \ast},\bm{Z}_{t_k}^\pi)$, and $\tilde{\bm{\alpha}}_{t_k}^{\pi,m} = \tilde{\bm{\alpha}}^{\pi, m}(t_k,\bm{X}_{t_k}^\pi)$, then $\tilde{\bm{\alpha}}_{t_k}^{\pi,m+1} = \bm a(t_k,\bm{X}_{t_k}^\pi,\bm{\alpha}_{t_k}^{\pi,m},\bm{Z}_{t_k}^{\pi,m+1})$. Thus, for any $\lambda > 0$, using the AM-GM inequality
\begin{align}
    \text{I} &\leq  (1 + \lambda^{-1}) \sum_{k=0}^{N_T - 1}e^{\eta t_k}\EE|\bm{\alpha}_{t_k}^{\pi,*} - \tilde{\bm{\alpha}}_{t_k}^{\pi,m+1}|^2\Delta t_k  + C(\lambda)\sum_{k=0}^{N_T - 1}e^{\eta t_k}\EE|\bm{\alpha}_{t_k}^{\pi,*} - \bm{\alpha}^*(t_k,\bm{X}_{t_k}^\pi)|^2\Delta t_k \notag \\
    &:= (1+\lambda^{-1})\text{I}^{(1)} + C(\lambda)\text{I}^{(2)}. \label{thm4_eq7_revise}
\end{align}
For term $\text{I}^{(1)}$, using \eqref{thm4_eq5_revise} and the Lipschitz condition of $\bm{a}$ in \eqref{eq:def_proj1}, we obtain 
\begin{align}\label{thm4_eq8_revise}
    \text{I}^{(1)}&\le \sum_{k = 0}^{N_T - 1}e^{\eta t_k}\left[L\EE\|\bm{Z}_{t_k}^\pi - \bm{Z}_{t_k}^{\pi,m+1}\|_F^2 + a_\alpha \EE|\bm{\alpha}_{t_k}^{\pi,*} - \bm{\alpha}_{t_k}^{\pi,m}|^2\right]\Delta t_k \notag \\
    & \le [a_\alpha(1+\lambda^{-1}) + \frac{C}{\eta -\beta_0}]\int_{0}^{T} e^{\eta t}\EE|\bm{\alpha}_t^{*} - \bm{\alpha}_{\pi(t)}^{\pi,m}|^2 \rmd t + C( \lambda,\eta) \left[\pinorm  + \EE|\bm{Y}_{T}^{\pi,m+1} - \bm{g}(\bm{X}_{T}^\pi)|^2 \right],
\end{align}
where we remove $C(\cdot)$'s dependence on $a_\alpha$ using $a_\alpha<1$ and we have also used
\begin{equation}
   \int_{0}^{T}e^{\eta t}\EE|\balpha_t^{*} - \balpha_{\pi(t)}^{\pi,*}|^2 = \int_{0}^{T}e^{\eta t}\EE|\balpha(t,\bX_t,\bZ_t) - \balpha(\pi(t),\bX_{\pi(t)}^\pi,\bZ_{\pi(t)}^\pi|^2 \rmd t \le C(\eta)\pinorm.
\end{equation}
Combining the last inequality with \eqref{thm4_eq6_revise} yields the estimate for term $\text{I}^{(2)}$:
\begin{equation}\label{thm4_eq9_revise}
   \text{I}^{(2)} \le C(\eta)\pinorm.
\end{equation}
Now plugging the estimates of $\text{I}^{(1)}$ and $\text{I}^{(2)}$ into \eqref{thm4_eq7_revise} and following step~\eqref{alpha_express_revise}--\eqref{alpha_time_revise}, we obtain:
\begin{align}
    &\quad \int_{0}^{T}e^{\eta t}\EE|\bm{\alpha}_t^{*} - \bm{\alpha}_{\pi(t)}^{\pi,m+1}|^2 \rmd t \notag\\
    & \leq (1+\lambda^{-1}) \int_{0}^{T}e^{\eta t}\EE|\bm{\alpha}^{*} - \balpha^{\pi,m+1}|^2(\pi(t),\bX_{\pi(t)}^\pi) \ud t + C(\lambda) \text{III}\\
    &\leq  (1+\lambda^{-1})e^{\eta \pinorm} \text{II} + C(\lambda) \text{III} \leq (1+\lambda^{-1})\text{II} + C(\lambda)\text{III} + C(\lambda,\eta)\pinorm\\
    &\leq  (1+ \lambda^{-1})^2\text{I}^{(1)} + C(\lambda)(\text{I}^{(2)} + \text{III}) + C(\lambda,\eta)\pinorm \notag \\
    &\le (1+\lambda^{-1})^2[a_\alpha (1+\lambda^{-1}) + \frac{C}{\eta - \beta_0}]\int_{0}^{T}\EE|\bm{\alpha}_t^{*} - \bm{\alpha}_{\pi(t)}^{\pi,m}|^2 \rmd t  + C(\lambda,\eta) \left[\|\pi\| + \EE|\bm{Y}_{T}^{\pi,m+1} - \bm{g}(\bm{X}_{T}^\pi)|^2 \right],
\end{align}
where we have used that
\begin{equation}
    e^{\eta \pinorm}\text{II} \le \text{II} + C(\eta)\pinorm \text{II} \le \text{II} + C(\eta)\pinorm \sum_{k=0}^{N_T - 1}(\EE|\bX_{t_k}^\pi|^2+1)\Delta t_k \le \text{II} +C(\eta)\pinorm.
\end{equation}
Let $\lambda$ and $\eta_\epsilon$ be large enough such that 
\begin{equation}
    (1+\lambda^{-1})^2[a_\alpha(1+\lambda^{-1}) +\frac{C}{\eta_\epsilon - \beta_0}]
    \le a_\alpha + \epsilon,
\end{equation}
then for $\eta > \eta_\eps$  we deduce
\begin{equation}
    \int_{0}^{T}e^{\eta t}\EE|\balpha_t^{*} -\balpha_{\pi(t)}^{\pi,m}|^2\rmd t \le C(\eta,\epsilon)[\pinorm + (a_\alpha + \epsilon)^m\int_{0}^Te^{\eta t}\EE|\balpha_t^{*}-\balpha_{\pi(t)}^{\pi,0}|^2\rmd t + \sum_{j=1}^m(a_\alpha+\epsilon)^{m-j}\EE|\bm{g}(\bX_T^\pi) - \bY_T^{\pi,j}|^2].
\end{equation}
Combining the above inequality with inequality \eqref{thm4_eq4_revise} and \eqref{thm4_eq4'_revise}, we obtain our result.
\end{proof}

Here are some remarks regarding Theorem~\ref{thm_bsde_error_revise} on its implication for numerical algorithms. The primary concern is how we can implement the projection mapping in practice if wished. Note that we choose 1/2-H\"older continuity in time in Assumption~\ref{assumption gradientv} for the generality of the result, although numerically it is challenging to guarantee the H\"older continuity.
If we replace that with the Lipschitz continuity in time, as a more restrictive condition, and instead consider the projection onto the space with the Lipschitz continuity, the estimates still hold. Accordingly, there are some practical approaches in the literature on ensuring the Lipschitz continuity of deep neural networks that can be introduced in our algorithms.
For instance, \cite{fazlyab2019efficient} gives an efficient and accurate estimation of Lipschitz constants for deep Neural networks, and \cite{pauli2020training} further extends it for robust training with regularization to keep the Lipschitz constant of neural networks small. In practice, Wasserstein GAN~\cite{arjovsky17wgan,gulrajani2017improved} has shown remarkable performance when using weight clipping as a loose but efficient way to impose the Lipschitz constraint. Therefore we can leverage similar techniques to keep the Lipschitz regularity during the training of the deep fictitious play. Also, notice that in the above, we define a single projection $\cN^\pi$ from the space of all players' strategies $\bm{\alpha}^{\pi,m}$, in consideration of the simplicity of the statement. One can also use the projection of ${\alpha}^{i,\pi,m}$ for each player with possibly easier numerical implementation and the same theoretical guarantee.

\subsection{On the $\epsilon$-Nash Equilibrium}\label{sec_analysis_step3}

This section combines the previous analysis, identifies the $\eps$-Nash equilibrium produced by the deep fictitious play, and evaluates its performance on the original game.

\begin{thm}\label{thm_epsnash}
Under Assumptions \ref{assumption bsigma}--\ref{assumption gradientv}, if $\hat \balpha$ is a policy function on $[0,T]\times \RR^n$ and Lipschitz in $\bx$, 
and 
\begin{equation}
\label{eq:eval_alpha_cond2}
 \int_{0}^{T}\EE|\bm{\alpha}^\ast - \hat \balpha|^2(t, \bX_t) \rmd t \leq \eps,
\end{equation}
where $\bX_t$ is the forward component of \eqref{def_BSDE}, then
\begin{enumerate}[(1)]
    \item Given $\hat\balpha$, the game values produced by $\hat\balpha$ are near the Nash equilibrium, i.e.,
    \begin{equation}\label{eq_Jrelation}
        |\tilde {\bm J}_0(\hat\balpha)- \bm J_0(\balpha^\ast)|^2 \leq C \eps, \text { and } |\bm J_0(\hat\balpha) -\bm J_0(\balpha^\ast)|^2 \leq C\eps,
    \end{equation}
    where $\tilde {\bm J}_0(\hat \balpha) = [\tilde J_0^1(\hat\balpha), \ldots, \tilde J_0^N(\hat\balpha)]$ with $\tilde J_0^i(\hat\balpha) := \inf_{\beta^i\in \mathbb{A}^i} J_0^i(\beta^i, \hat\balpha^{-i})$,
    $\bm J_0(\hat\balpha) = [J_0^1(\hat\balpha), \ldots, J_0^N(\hat\balpha)]$ with $J_0^i$ defined in \eqref{def_J}. Thus, there exists $0 < \eps_i \ll 1$ such that $\sum_{i=1}^N \eps_i^2 \le C\eps$ and 
    \begin{equation}\label{eq_espnash}
        J_0^i(\beta^i, \hat\balpha^{-i}) \geq  J_0^i(\hat \balpha) - \eps_i, \quad \forall \beta^i \in \mathbb{A}^i \text{ and } i \in \mc{I}.
    \end{equation}
    Here $C$ is a constant depending on $T$, $L$, $M$ and $M'$ which may vary from line to line in the proof. 
     \item The generated game paths $\bX_t^{\hat \balpha}$ are close to the paths $\bX_t^{\balpha^\ast}$ associated with the Nash equilibrium:
    \begin{equation}
        \EE\left[\left(\sup_{0 \leq t \leq T}\left|\bX_t^{\balpha^\ast} - \bX_t^{\hat\balpha}\right|\right)^2\right] \leq C(\lambda) \eps^\lambda,
    \end{equation}
    where $\bX_t^{\balpha^\ast}$ and $\bX_t^{\hat \balpha}$ follow \eqref{def_Xt} with the true Nash equilibrium strategy $\balpha^\ast$ and $\hat \balpha$. Here $\lambda$ is an arbitrary constant in $(0, 1)$, and $C(\lambda)$ is a constant depending on $T$, $L$, $M$, $M'$ and $\lambda$.
\end{enumerate}

\end{thm}

Immediately, we have the following corollary.
\begin{cor}
Under Assumptions~\ref{assumption bsigma}--\ref{assumption gradientv}, assuming the sub-problems \eqref{eq_Xdiscrete} are solved accurate enough at all stages, i.e.,
\begin{equation}\label{assump_epsnash}
    \EE|\bm g(X_T^\pi) - \bY_T^{\pi, j}|^2 \leq C \eps^2, \quad  \forall j \leq m,
\end{equation}
here C is a constant depending only on $T$, $L$, $M$, $M'$, $K$, $L'$ and $\EE|\bx_0|^2$.
Then, for sufficiently large $m$ and small mesh size $\pinorm$, the strategy $\balpha^m$ defined in \eqref{alpha_interpolation_revise}, as an interpolated policy based on the deep fictitious play, forms an $\eps$-Nash equilibrium.
\end{cor}

\begin{proof}
This follows from \eqref{eq_espnash} in Theorem~\ref{thm_epsnash}, with the assumptions satisfied according to equations \eqref{eq_deep_bsde_error_revise}, \eqref{alpha_m_regularity_revise} and \eqref{thm4_eq3_revise}. 
\end{proof}

\begin{rem}
As mentioned in Remark~\ref{rem_thm1}, there are still some theoretical issues unsolved regarding the approximation error and optimization of the deep BSDE method. The analysis of the deep fictitious play method has similar issues that remain open. To circumvent these issues and have a rigorous statement for $\eps$-Nash equilibrium, we introduce assumption~\eqref{assump_epsnash}. In practice, an observable proxy of \eqref{assump_epsnash} is the training loss of the deep BSDE method evaluated by its Monte Carlo counterpart.
\end{rem}

\begin{proof}[Proof of Theorem~\ref{thm_epsnash}]

The proof of item (1) relies on the estimates of BSDEs presented previously.
Let $(\bX_t, Y_t^{i, \emph{FP}}, Z_t^{i, \emph{FP}})$ solve \eqref{eq:FP_BSDE} with $\balpha^{-i,m}$ replaced by $\hat\balpha^{-i}$.
By the nonlinear Feynman-Kac formula (cf. \cite{PaPe:92,ElPeQu:97,PaTa:99}) and the associated HJB equation,
we have $\EE[Y_0^{i, \text{FP}}] = \tilde J_0^i(\hat\balpha)$.
Therefore, we have
\begin{equation}
\label{eq_Jrelation1}
     |\tilde {\bm J}_0(\hat\balpha)- \bm J_0(\balpha^\ast)|^2 = |\EE[\bm Y_0^\text{FP}] - \EE[\bm Y_0] |^2 \leq \EE|\bm Y_0^\text{FP} - \bm Y_0|^2.
\end{equation}
To bound the above term, we claim a stronger result:
\begin{equation}\label{eq_prenash}
    \sup_{0 \leq t \leq T} \EE |\bY_t^{\text{FP}} - \bY_t|^2 + \int_0^T \EE\|\bZ_t^{\text{FP}} - \bZ_t\|_F^2 \ud t \leq C \eps,
\end{equation}
where $(\bY_t, \bZ_t)$ solves \eqref{def_BSDE}, $\bY_t^\text{FP} = [Y_t^{1, \text{FP}}, \ldots, Y_t^{N, \text{FP}}]\transpose$, and $\bZ_t^\text{FP} = [Z_t^{i, \text{FP}}, \ldots, Z_t^{N, \text{FP}}]$, as a consequence of \eqref{Y_control_thm4_revise}, \eqref{Z_control_thm4_revise} and \eqref{eq:eval_alpha_cond2}.


If we let $(\bX_t, Y_t^{i, \emph{PU}}, Z_t^{i, \emph{PU}})$ solve \eqref{eq:PU_BSDE} with $\balpha^m$ replaced by $\hat \balpha$,
an argument similar to \eqref{eq_Jrelation1} and \eqref{eq_prenash} can give the second inequality in \eqref{eq_Jrelation}. 
Then \eqref{eq_espnash} is obtained by observing 
\begin{equation}
    J_0^i(\beta^i, \hat\balpha^{-i}) \geq \tilde J_0^i(\hat \balpha),  \quad  \forall \beta^i \in \mathbb{A}^i,
\end{equation}
and $|\tilde {\bm J}_0(\hat\balpha)- \bm J_0(\hat \balpha)|^2 \leq C\eps.$

We now prove item (2). Under the standing assumptions, we first observe that $b^1(t, \bx) := b(t, \bx, \balpha^
\ast(t, \bx)) = \Sigma(t, \bx) \phi(t, \bx, \balpha^\ast(t, \bx))$ and $b^2(t, \bx) := b(t, \bx, \hat\balpha(t, \bx)) = \Sigma(t, \bx) \phi(t, \bx, \hat\balpha(t, \bx))$  are Lipschitz in $\bx$. Thus $\bX_t^{\hat \balpha}$ is well-defined, and the standard estimates in SDE gives (cf. \cite[Theorem 3.2.4]{Zh:17})
 \begin{equation}\label{thm5_eq1}
        \EE\left[\left(\sup_{0 \leq t \leq T}\left|\bX_t^{\balpha^\ast} - \bX_t^{\hat\balpha}\right|\right)^2\right] \leq C \EE\left[\left(\int_0^T b^1(t, \bX_t^{\balpha^\ast}) - b^2(t, \bX_t^{\balpha^\ast}) \ud t\right)^2\right].
    \end{equation}
    
To bound the right-hand side above with the condition \eqref{eq:eval_alpha_cond2}, let us define a new probability measure $\QQ$, and denote by $\ZZ$ the Radon-Nikodym derivative:
\begin{equation}
    \frac{\ud\QQ}{\ud \PP} \equiv \ZZ := \exp\left\{-\int_0^T \phi_t^{\balpha^\ast} \cdot \ud \bW_t - \half \int_0^T |\phi_t^{\balpha^\ast}|^2\ud t \right\},
\end{equation}
where $\phi_t^{\balpha^\ast}:= \phi(t, X_t^{\balpha^\ast}, \balpha^\ast(t, X_t^{\balpha^\ast}))$.
By Assumptions~\ref{assumption 1 revise}, the Novikov's condition is fulfilled. Thus $\QQ \sim \PP$, and $\bm W_\cdot^\QQ := \bm W_\cdot + \int_0^\cdot \phi_t^{\balpha^\ast} \ud s$ is a standard Brownian motion under $\QQ$. In particular, the process $\bX_t^{\balpha^\ast}$ can be rewritten as $ \bX_t^{\balpha^\ast} = \bx_0 + \int_0^t \Sigma(s, \bX_s^{\balpha^\ast}) \ud \bW_s^\QQ$, and immediately from \eqref{eq:eval_alpha_cond2} we have
\begin{equation}\label{eq_XastestimateunderQ}
 \int_{0}^{T}\EE_\QQ|\bm{\alpha}^\ast - \hat \balpha|^2(t, \bX_t^{\balpha^\ast}) \rmd t \leq \eps,
\end{equation}
where we denote by $\EE_\QQ$ the expectation under measure $\QQ$. We next compute a bound for $\ZZ^{-\gamma}$ under $\QQ$, for $\gamma > 1$:
\begin{align}
    \EE_\QQ[\ZZ^{-\gamma}] & = \EE_\QQ\left[\exp\left\{\gamma\int_0^T \phi_t^{\balpha^\ast} \cdot \ud \bW_t^\QQ - \frac{\gamma}{2} \int_0^T |\phi_t^{\balpha^\ast}|^2\ud t \right\}\right] \notag\\
    & \leq   \EE_\QQ^{1/2} \left[\exp\left\{2\gamma\int_0^T \phi_t^{\balpha^\ast} \cdot \ud \bW_t^\QQ - 2\gamma^2 \int_0^T |\phi_t^{\balpha^\ast}|^2\ud t \right\}\right] \times \EE_\QQ^{1/2}\left[\exp\left\{ (2\gamma^2-\gamma) \int_0^T |\phi_t^{\balpha^\ast}|^2\ud t \right\}\right] \notag \\
    & \leq e^{CT(\gamma^2 -\half \gamma)},
\end{align}
where $\EE^p_\QQ$ denote $(\EE_\QQ[\cdot])^p$,  and we have used the Cauchy-Schwartz inequality and the boundedness of $\phi$ in Assumptions~\ref{assumption 1 revise}.
Therefore, we have
\begin{align}
    &\quad~\EE\left[\left(\int_0^T b^1(t, \bX_t^{\balpha^\ast}) - b^2(t, \bX_t^{\balpha^\ast}) \ud t\right)^2\right] \notag \\
    & \leq \EE^{1-\frac{1}{\gamma}}_\QQ\left[\left(\int_0^T b^1(t, \bX_t^{\balpha^\ast}) - b^2(t, \bX_t^{\balpha^\ast}) \ud t\right)^{\frac{2\gamma}{\gamma-1}}\right] \EE_\QQ^{\frac{1}{\gamma}}[\mc{Z}^{-\gamma}] \notag\\
    & \leq C(\gamma) \EE_\QQ^{1-\frac{1}{\gamma}}\left[\left(\int_0^T b^1(t, \bX_t^{\balpha^\ast}) - b^2(t, \bX_t^{\balpha^\ast}) \ud t\right)^2\right] \EE_\QQ^{1-\frac{1}{\gamma}}\left[\left(\int_0^T b^1(t, \bX_t^{\balpha^\ast}) - b^2(t, \bX_t^{\balpha^\ast}) \ud t\right)^\frac{2\gamma+2}{\gamma-1}\right] \notag \\
    & \leq  C(\gamma) \EE_\QQ^{1-\frac{1}{\gamma}}\left[\int_0^T |\balpha^\ast - \hat\balpha|^2(t, \bX_t^{\balpha^\ast}) \ud t\right] \leq C(\gamma)\eps^{1-\frac{1}{\gamma}}, \notag
\end{align}
where we have consecutively used H\"{o}lder's inequality, the estimate of $\EE_\QQ[\mc{Z}^{-\gamma}]$, the Lipschitz property of $\phi(t, \bx, \balpha)$, the boundedness of $\Sigma$ and $b$, and the estimate \eqref{eq_XastestimateunderQ}. Here $C(\gamma)$ is a constant depending on the $T$, $L$, $M$, $M'$ and $\gamma$, which may vary from line to line. With \eqref{thm5_eq1} and noticing $0 < 1-\frac{1}{\gamma} < 1$ we conclude. 
\end{proof}

In practice, the game is play on $\mc{T}$, but not $[0,T]$. 
Therefore, we will define a discrete version of the stochastic differential game \eqref{def_Xt}--\eqref{def_obj} and evaluate the performance of $\bm{\alpha}^{\pi,m}$ in section \ref{sec_step2} on the discrete game. To be precise, given a policy function $\bm{\alpha}^\pi$ on $\mc{T}\times \RR^n$, we define the discrete state process $\bm{X}_{t_k}^{\pi,\bm{\alpha}^\pi}$ and discrete individual cost functional $J_0^{\pi,i}(\bm{\alpha}^\pi)$ as follows
\begin{align}\label{eq_discreteX}
    &\bm{X}_{t_{k+1}}^{\pi,\bm{\alpha}^\pi} = \bm{X}_{t_k}^{\pi,\bm{\alpha}^\pi} + b(t_k,\bm{X}_{t_k}^{\pi,\bm{\alpha}^\pi},\bm{\alpha}^\pi(t_k,\bm{X}_{t_k}^{\pi,\bm{\alpha}^\pi}))\Delta t_k +\Sigma(t_k,\bm{X}_{t_k}^{\pi,\bm{\alpha}^\pi})\Delta \bm{W}_k, \quad \bm{X}_0^{\pi,\bm{\alpha}^{\pi}} = \bm{x}_0, \\
    &J_0^{\pi,i}(\bm{\alpha}^\pi) = \EE\left[\sum_{j = 0}^{N_T-1} f^i(t_k,\bm{X}_{t_k}^{\pi,\bm{\alpha}^\pi},\bm{\alpha}^\pi(t_k,\bm{X}_{t_k}^{\pi,\bm{\alpha}^\pi})) \Delta t_k + g^i(\bm{X}_{T}^{\pi,\bm{\alpha}^\pi})\right].
\end{align}
Note that when there are both $\pi$ and $\bm{\alpha}$ in the superscript of $\bm{X}$, it refers to the (discrete) process of the original state~\eqref{def_Xt}, and when there is only $\pi$ in the superscript, it refers to the (discrete) process $\bX_t =  \bx_0 + \int_0^t \Sigma(s, \bX_s) \ud \bm W_s.$ We then state a discrete version of Theorem \ref{thm_epsnash}.

\begin{thm}
Under Assumptions \ref{assumption bsigma}--\ref{assumption gradientv}, if $\hat \balpha^{\pi}$ is a policy function on $\mathcal{T}\times \RR^n$, Lipschitz in $\bx$ and H\"older continuous with t:
\begin{equation}
    |\hat{\bm{\alpha}}^{\pi}(t_1,\bm{x}_1) - \hat{\bm{\alpha}}^{\pi}(t_2,\bm{x}_2)|^2 \le L'[|t_1 - t_2| + |\bm{x}_1 - \bm{x}_2|^2],
\end{equation}
and 
\begin{equation}\label{eq:eval_alpha_cond3}
 \int_{0}^{T}\EE|\bm{\alpha}^\ast(t,\bm{X}_t) - \hat \balpha^{\pi}(\pi(t),\bm{X}_{\pi(t)}^{\pi})|^2 \rmd t \leq \eps,
\end{equation}
then
\begin{enumerate}[(1)]
    \item The value of the discrete game produced by $\hat{\bm{\alpha}}^{\pi}$ is close to the
    one associated with the Nash equilibrium of the continuous game, i.e.,
    \begin{equation}
        |\bm J_0^{\pi}(\hat\balpha^\pi) -\bm J_0(\balpha^\ast)|^2 \leq C[\eps+\|\pi\|],
    \end{equation}
    where $\bm J_0^{\pi}(\hat\balpha^{\pi}) = [J_0^{\pi,1}(\hat\balpha^{\pi}), \ldots, J_0^{\pi,N}(\hat\balpha^{\pi})]$. Moreover, there exists $0 < \eps_i \ll 1$ such that $\sum_{i=1}^N \eps_i^2 \leq C[\eps + \|\pi\|]$ and 
    \begin{equation}
        J_0^i(\beta^i, \hat\balpha^{\pi,-i}) \geq  J_0^{\pi,i}(\hat \balpha^{\pi}) - \eps_i, \quad \forall \beta^i \in \mathbb{A}^i \text{ and } i \in \mc{I}.
    \end{equation}
    Here $C$ is a constant depending on $T$, $L$, $M$, $M'$, $K$, $L'$ and $\EE|\bx_0|^2$, which may vary from line to line in the proof.
     \item The generated game paths $\bX_{t_k}^{\pi, \hat\balpha^\pi}$ are close to the paths $\bX_t^{\balpha^\ast}$ associated with the Nash equilibrium:
    \begin{equation}
        \EE\left[\left(\sup_{0 \leq t \leq T}\left|\bX_t^{\balpha^\ast} - \bX_{\pi(t)}^{\pi,\hat{\balpha}^\pi}\right|\right)^2\right] \leq C(\lambda) [\eps + \|\pi\|]^\lambda,
    \end{equation}
    where $\bX_t^{\balpha^\ast}$ follow \eqref{def_Xt} with the true Nash equilibrium strategy $\balpha^\ast$ and $\bX_{t_k}^{\pi,\hat{\balpha}^\pi}$ follow \eqref{eq_discreteX}. Here $\lambda$ is an arbitrary constant in $(0,1)$, and $C(\lambda)$ is a constant depending on $T$, $L$, $M$, $M'$, $K$, $L'$ and $\EE|\bx_0|^2$ and $\lambda$.
\end{enumerate}

\end{thm}

\begin{proof}
Let $\hat{\balpha}(t,\bm{x}) = \inf_{t'\in \mathcal{T}}[\hat{\balpha}^{\pi}(t',\bm{x}) + L'|t'-t|^{\frac{1}{2}}]$, then with an argument similar to that in Theorem \ref{thm_bsde_error_revise}, $\hat{\balpha}$ satisfies:
\begin{equation}\label{alpha_regularity_thm6}
    |\hat{\bm{\alpha}}(t_1,\bm{x}_1) - \hat{\bm{\alpha}}(t_2,\bm{x}_2)|^2 \le L'[|t_1 - t_2| + |\bm{x}_1 - \bm{x}_2|^2].
\end{equation}
By \eqref{X_discrete_error_revise}, \eqref{eq:eval_alpha_cond3} and \eqref{alpha_regularity_thm6}, we have
\begin{align}\label{alpha_hat_distance}
    &\quad\int_{0}^{T}\EE|\bm{\alpha}^{*} - \bm{\hat{\alpha}}|^2(t,\bm{X}_t) \rmd t
    \notag \\
    &\le 2\left[\int_{0}^{T}\EE|\bm{\alpha}^\ast(t,\bm{X}_t) - \hat \balpha^{\pi}(\pi(t),\bm{X}_{\pi(t)}^{\pi})|^2 \rmd t + \int_{0}^{T}\EE|\bm{\hat{\alpha}}(t,\bm{X}_t) - \hat \balpha^{\pi}(\pi(t),\bm{X}_{\pi(t)}^{\pi})|^2 \rmd t\right] \notag \\
    &\le C[\|\pi\| + \eps].
\end{align}
By the regularity of $\hat{\bm{\alpha}}$ (c.f. \eqref{alpha_regularity_thm6}) and the standard estimates of the Euler Scheme of SDE (c.f. \cite{kloeden2013numerical}), we can obtain
\begin{equation}\label{X_alpha_discrete}
    \EE\left[\left(\sup_{0 \leq t \leq T}\left|\bX_t^{\hat{\balpha}} - \bX_{\pi(t)}^{\pi,\hat{\balpha}^\pi}\right|\right)^2\right] \le C\|\pi\|.
\end{equation}
Observing that
\begin{align}
    &\quad \bm{J}_0^{\pi}(\hat{\bm{\alpha}}^\pi) - \bm{J}_0(\hat{\bm{\alpha}}) \notag\\
    & = \EE\left[\int_{0}^{T}[\bm{f}(\pi(t),\bX_{\pi(t)}^{\pi,\hat{\balpha}^\pi},\bm{\hat{\alpha}}^{\pi}(\pi(t),\bX_{\pi(t)}^{\pi,\hat{\balpha}^\pi})) - \bm{f}(t,\bm{X}_t^{\hat{\bm{\alpha}}},\hat{\bm{\alpha}}(t,\bm{X}_t^{\hat{\bm{\alpha}}}))]\rmd t + \bm{g}(\bX_{T}^{\pi,\hat{\balpha}^\pi}) - \bm{g}(\bX_T^{\hat{\balpha}})\right],
\end{align}
with \eqref{alpha_regularity_thm6}, \eqref{X_alpha_discrete} and Assumption \ref{assumption 1 revise}, one has
\begin{equation}\label{J_distance_discrete}
    |\bm{J}_0^{\pi}(\hat{\bm{\alpha}}^\pi) - \bm{J}_0(\hat{\bm{\alpha}})|^2 \le C\|\pi\|. 
\end{equation}
Finally, with \eqref{alpha_hat_distance}, \eqref{X_alpha_discrete}, \eqref{J_distance_discrete} and Theorem \ref{thm_epsnash}, we reach all the conclusions of this theorem.
\end{proof}

\section{Numerical Examples}\label{sec_numerics}

We supplement our theoretical analysis by numerical examples. We shall mainly focus on how deep BSDE performs when combined with policy update strategy in the decoupling step, {\it i.e.}, when solving \eqref{eq:PU_BSDE}. We refer readers to \cite{HaHu:19} for the numerical performance when the deep BSDE method is used to solve the sub-problems derived from fictitious play, whose results are similar to those presented here from policy update.
The example we present here is an inter-bank game concerning the systemic risk \cite{CaFoSu:15}.
Assume an inter-bank market with $N$ banks,
and denote by $X^i_t\in \RR$ the log-monetary reserves of bank $i$ at time $t$. Its dynamics are modeled as the following diffusion processes,
\begin{equation}
\label{eq:ex1_dynamics}
\ud X_t^i = [a(\overline X_t - X_t^i)   + \alpha_t^i ] \ud t  + \sigma \left(\rho \ud W_t^0 + \sqrt{1-\rho^2} \ud W_t^i\right), \quad \overline X_t = \frac{1}{N}\sum_{i=1}^N X_t^i, \quad i \in \mc{I}.
\end{equation}
Here $a(\overline X_t - X_t^i)$ represents the rate at which bank $i$ borrows from or lends to other banks in the lending market, while $\alpha_t^i$  denotes its control rate of cash flows to a central bank. The standard Brownian motions $\{W_t^i\}_{i=0}^N$ are independent, in which $\{W_t^i, i\geq 1\}$ stands for the idiosyncratic noises and $W_t^0$ denotes the systemic shock, or so-called common noise in the general context. To describe the model in the form of \eqref{def_Xt}, we concatenate the log-monetary reserves $X_t^i$ of $N$ banks to form $\bX_t^{\balpha}=[X_t^1,\dots,X_t^N]\transpose$.
The associated drift term and diffusion term are defined as
\begin{equation}
\label{eq:example1_driftdiff}
b(t, \bm x, \balpha)=
[a(\bar x - x^1) + \alpha^1, 
\ldots, 
a(\bar x - x^N) + \alpha^N
]\transpose\in \RR^{N\times1}
, \quad \bar x = \frac{1}{N}\sum_{i=1}^N x^i, \end{equation}
\begin{equation}
\label{eq:example1_sigma}
\Sigma(t, \bm x)=
\begin{bmatrix}
     \sigma\rho & \sigma\sqrt{1-\rho^2} & 0 & \cdots & 0\\
     \sigma\rho & 0 & \sigma\sqrt{1-\rho^2} & \cdots & 0\\
     \vdots & \vdots & \vdots & \ddots & \vdots\\
     \sigma\rho & 0 & 0 & \cdots &\sigma\sqrt{1-\rho^2}
\end{bmatrix}\in \RR^{N\times (N+1)},
\end{equation}
and $\bW_t = (W_t^0, \ldots, W_t^N)$ is $(N+1)$-dimensional. The cost functional \eqref{def_J} that player $i$ wishes to minimize has the form
\begin{equation}
f^i(t, \bm{x},\balpha) = \half (\alpha^i)^2 - q \alpha^i(\bar x - x^i) + \frac{\epsilon}{2}(\bar x - x^i)^2,  \quad g^i(\bm{x}) = \frac{c}{2}(\bar x - x^i)^2.
\end{equation}
Under such specifications, the solution of this game admits a quadratic form whose coefficient functions can be solved from a Riccati equation.
We direct the interested readers to \cite{CaFovosSu:15,HaHu:19} for the detailed interpretation of this model and the explicit characterization of the solution.

\begin{figure}[!ht]
    \centering
    \includegraphics[width=0.95\textwidth]{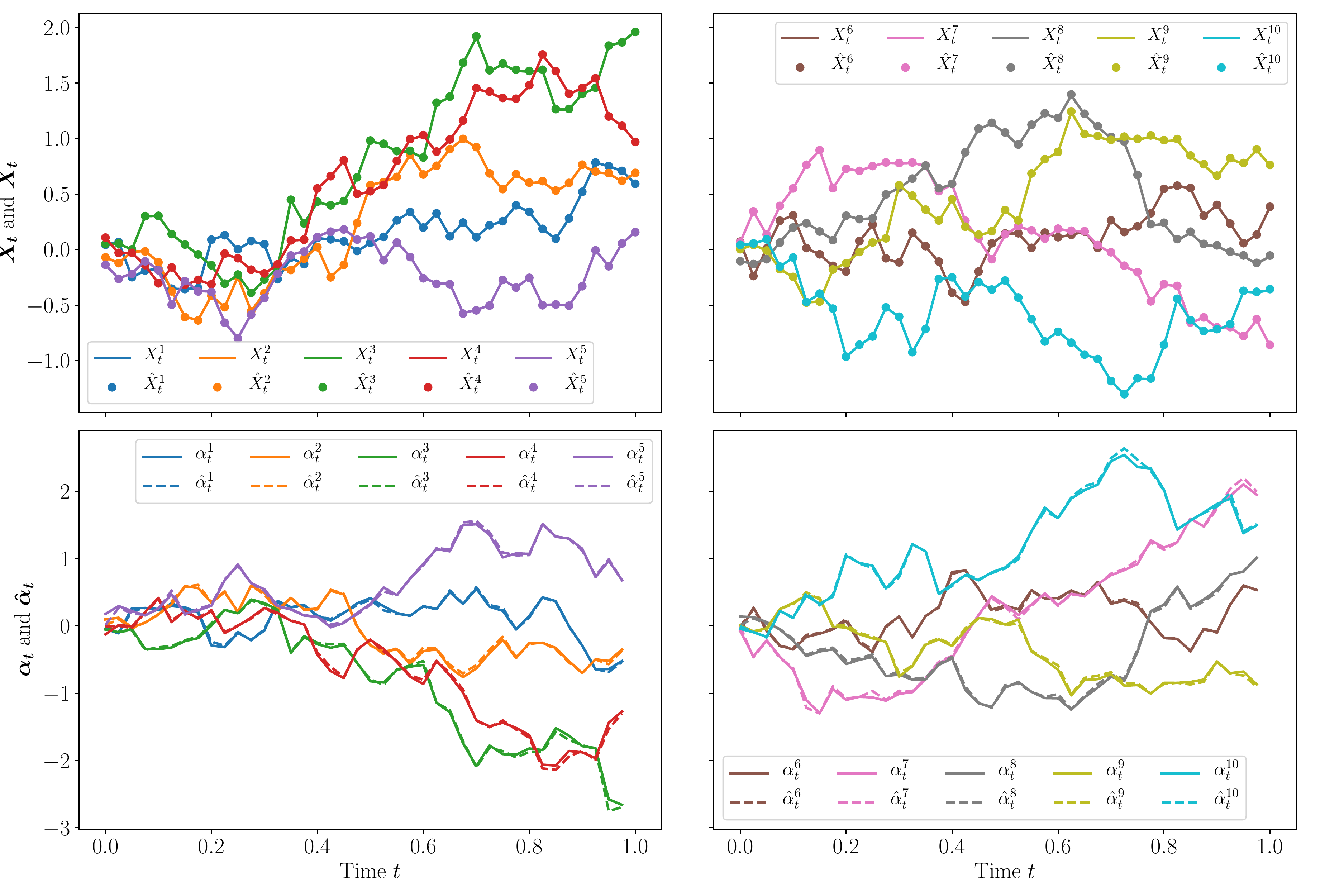}
    \caption{A sample path for each player of the inter-bank game when $N=10$, obtained from decoupling the problem by policy update and solving the sub-problems with the Deep BSDE method. Top: the optimal state process $X_t^i$ (solid lines) and its approximation $\hat{X}_t^i$ (circles) provided by the optimized neural networks, under the same realized path of Brownian motion. Bottom: comparisons of the strategies $\alpha_t^i$ and $\hat{\alpha}_t^i$ (dashed lines).}
    \label{fig:bsde}
\end{figure}

In our numerical computation, we choose $N=10$, $T=1$, and
\begin{equation}\label{def_parameters}
    a = 0.1,\quad  q = 0.1, \quad c = 0.5,\quad  \epsilon = 0.5,\quad  \rho = 0.2, \quad \sigma = 1.
\end{equation}
We discretize the time $[0, T]$ into $N_T=40$ intervals and specify the hypothesis spaces $\cN_0^{i'}$ and $\{\cN_k^i\}_{k=0}^{N_T-1}$ for each player $i$ as follows.
We parametrize $V^{i}(0,\bx)$ (the superscript $m$ is dropped again for simplicity) with a neural network, denoted by $\text{Net}_{V}(\bx)$, as the space $\cN_0^{i'}$ of $Y_0^i$.
We also parametrize $\nabla_{\bm{x}}V^{i}(t,\bx)$ with another network, denoted by $\text{Net}_{\nabla V}(t,\bx)$,
as the space of $\{\cN_k^i\}_{k=0}^{N_T-1}$, in which the timestamp $t_k$ is provided as another dimension of the input. This choice is in consistence with our theoretical analysis in Theorem~\ref{thm_bsde_error_revise} involving the linear interpolation of the strategy in time.

In this numerical example, we use fully-connected feedforward networks to instantiate both $\text{Net}_{V}(\bx)$ and $\text{Net}_{\nabla V}(t,\bx)$.
Since the problem is homogeneous among all the players, we let two networks share the same parameters among all the players and only solve player 1's problem for updating the parameters.
Both networks consist of three hidden layers with a width of $40$. The activation function is hyperbolic tangent, and the technique of batch normalization~\cite{batchnorm15} is adopted right after each linear transformation and before activation. For simplicity, we do not impose the projection procedure discussed in Section~\ref{sec_step2}.

Regarding the optimization, the loss function in Deep BSDE is differentiable with respect to the network parameters. We can use backpropagation to derive the gradient of the loss function with respect to all the parameters in the neural networks and employ stochastic gradient descent (SGD) to optimize all the parameters. 
In this work, we use Adam optimizer~\cite{Kingma2015adam} to optimize network parameters with constant learning rate 5e-4 and batch size $256$. The parameters are updated by 30000 steps in total.

To implement the algorithm, we also need to specify the distribution of the initial state $\bx_0$ in \eqref{eq:FP_BSDE} or \eqref{eq:PU_BSDE}. We follow the same way as in~\cite{HaHu:19}. Each component of $\bx_0$, as the initial state of each player, is sampled independently from the uniform distribution on $[-\delta_0, \delta_0]$. $\delta_0$ is chosen such that in the following process driven by the optimal policy $\balpha^*$
\begin{equation*}
\ud \bm X_t^{\bm \alpha^*} = b(t, \bm X_t^{\bm \alpha^*}, \bm \alpha^*(t, \bm X_t^{\bm \alpha^*})) \ud t + \Sigma(t, \bm X_t^{\bm \alpha^*}) \ud \bm W_t, \quad \bm X_0 = \bx_0,
\end{equation*}
the standard deviation of $\{\bX_t\}_{t=0}^T$ is approximately $\delta_0$. In other words, $\delta_0$ is determined as a fixed-point.
The rationale for such a procedure is to make sure the data generated for the learning is representative enough in the whole state space. 

Note that our technical assumptions are not strictly satisfied in this example, since $\bm{f}, \bm{g}$ are not Lipschitz continuous, $\phi$ is not uniform bounded, and $T$ is not sufficiently small.
Nevertheless, our numerical results show that the deep BSDE method can solve this game when combined with policy update. In particular, we compute the relative error of controls (proportional to the gradient of value function):
\begin{align*}
    \text{RSE} = \frac{\sum_{\substack{0 \leq k \leq N_T-1 \\ 1 \leq j \leq J}} \left(\nabla_{\bx} V^1(t_k, \bm x_{t_k}^{(j)}) - \nabla_{\bx} \widehat V^1(t_k,\bm x_{t_k}^{(j)})\right)^2}{\sum_{\substack{0 \leq k \leq N_T-1\\1 \leq j \leq J}} \left(\nabla_{\bx} V^1(t_k, \bm x_{t_k}^{(j)})- \overline {\nabla_{\bx} V}^1\right)^2},
\end{align*}
where 
$V^1$ is the true solution (of player 1),
$\hat{V}^1$ is the prediction from the neural networks,
and $\bar V^1 ~(\emph{resp.~} \overline {\nabla_{\bx} V}^1)$ is the average of $V^1 ~(\emph{resp.~} \nabla_{\bx}V^1)$ evaluated at all the indices $j, k$.
To compute the relative error, we generate $J=256$ ground truth sample paths $\{\bm x_{t_k}^{(j)}\}_{k=0}^{N_T-1}$ using Euler scheme based on  \eqref{eq:PU_BSDE} and the true optimal strategy. Note that the superscript ${(j)}$ here does not mean the player index, but the $j^{th}$ path for all players. The final RSE for the Deep BSDE method is
0.27\%.
Figures~\ref{fig:bsde} presents one sample path for each player of the optimal state process $X_t^i$ and the optimal control $\alpha_t^i$ \emph{vs.} their approximations $\hat{X}_t^i, \hat{\alpha}_t^i$, with good agreement.

\section{Conclusion}\label{sec_conclusion}
In this paper, we established the theoretical foundation for the deep fictitious play algorithm for finding Markovian Nash equilibrium proposed in \cite{HaHu:19}. Specifically, we proved the following three things: 1. The solutions of the decoupled sub-problems, if solved exactly and repeatedly, converge to the true Nash equilibrium; 2. The numerical error of each sub-problem, if solved by deep BSDE individually and repeatedly, converges to zero subject to the universal approximation capacity of neural networks; 3. The interpolated strategy based on the deep fictitious play algorithm forms a $\eps$-Nash equilibrium, after sufficiently many stages $m$ and with sufficiently small mesh $\pinorm$. We also generalize the algorithm by proposing a new approach to decouple the games, and present a numerical example in the end to show the empirical convergence beyond the technical assumptions used in the theorems. In the future, with this solidly established theory of deep fictitious play, we aim to study the competitions in Finance, including P2P lending platforms from the Fintech industry and insurance markets. We also plan to generalize the theory and algorithm to stochastic differential games with delays. 

\section*{Data availability statement}
Data sharing not applicable to this article as no new data were created or analysed during the this study.

\bibliographystyle{plain}
\bibliography{Reference}

\appendix
\section{Supporting Propositions for Assumption \ref{assumption lipschitz solution}}
We prove the following propositions in this section.
\begin{prop}
\label{prop:assump3_smalltime}
Under Assumptions \ref{assumption bsigma} and \ref{assumption 1 revise}, there exists a constant $T_0 > 0$ only depending on $L$ and $M$, such that Assumption \ref{assumption lipschitz solution} is satisfied when $T \le T_0$.
\end{prop}
\begin{prop}\label{prop:unique}
Under Assumptions \ref{assumption lipschitz solution}, the BSDE systems \eqref{def_BSDE} has a unique adapted solution satisfying that
\begin{equation}
    \EE \left[\sup_{0\le t \le T}(|\bm{X}_t|^2 +|\bm{Y}_t|^2) + \int_{0}^T \|\bm{Z}_t\|_F^2 \rmd t\right] < + \infty.
\end{equation}
\end{prop}

First, we need the following classical estimation (cf. \cite[Theorem~5.2.2]{zhang2017backward}), which give a uniform bound of the $Z$-component of BSDEs. We state and provide the proof of this lemma here to compute an exact upper bound for the convenience of later analysis.

\begin{lem}\label{lem1}
Let $(\bm{X}_t,\bm{Y}_t,\bm{Z}_t)$, the adapted process taking value in $\bR^n \times \bR^N \times \bR^{k\times N}$, be the solution of the following BSDE system:
\begin{equation}\label{lem1-BSDE}
    \begin{dcases}
    \bm{X}_t = \bm{x}_0 + \int_{0}^{t}\Sigma(s,\bm{X}_s)\rmd \bm{W}_s, \\
    \bm{Y}_t = \bm{g}(\bm{X}_T) + \int_{t}^T \bm{F}(s,\bm{X}_s,\bm{Z}_s)\rmd s - \int_{t}^T(\bm{Z}_s)\transpose \rmd \bm{W}_s,
    \end{dcases}
\end{equation}
where the coefficients $\Sigma$, $\bm{F}$ and $\bm{g}$ satisfy the following condition:
\begin{align*}
    |\bm{g}(\bm{x}_1) - \bm{g}(\bm{x}_2)|^2 &\le g_x|\bm{x}_1 - \bm{x}_2|^2, \\
    \|\Sigma(t,\bm{x}_1) - \Sigma(t,\bm{x}_2)\|_F^2 &\le \overline{L}|\bm{x}_1 - \bm{x}_2|^2, \\
    |\bm{F}(t,\bm{x}_1,\bm{p}_1) - \bm{F}(t,\bm{x}_2,\bm{p}_2)|^2 &\le \overline{L}[|\bm{x}_1 - \bm{x}_2|^2 + \|\bm{p}_1 - \bm{p}_2\|_F^2], \\
    \|\Sigma(t,\bm{x})\|^2_S &\le \overline{M}.
\end{align*}
Then:
\begin{equation}
    \|\bm{Z}_t\|^2_S \le \overline{M} (g_x+T)e^{2\overline{L}T} \quad \mathbb{L}\times\mathbb{P}\text{-a.s., }
\end{equation}
where 
$\mathbb{L}$ is the Lebesgue measure on $[0,T]$.
\end{lem}
\begin{proof}
Denote by $(\bm{X}^{t,\bm{x}}_s,\bm{Y}^{t,\bm{x}}_s,\bm{Z}^{t,\bm{x}}_s)_{ t \le s \le T}$ the solution of the following BSDE:
\begin{equation}\label{lem1-BSDE2}
    \begin{dcases}
    \bm{X}^{t,\bm{x}}_s = \bm{x} + \int_{t}^{s}\Sigma(u,\bm{X}^{t,\bm{x}}_u)\rmd \bm{W}_u, \\
    \bm{Y}^{t,\bm{x}}_s = \bm{g}(\bm{X}^{t,\bm{x}}_T) + \int_{s}^T \bm{F}(u,\bm{X}^{t,\bm{x}}_u,\bm{Z}^{t,\bm{x}}_u)\rmd u - \int_{t}^T(\bm{Z}^{t,\bm{x}}_u)\transpose \rmd \bm{W}_u,
    \end{dcases}
\end{equation}
for any $(t,\bm{x}) \in [0,T]\times \bR^n$. Then, for any $t_0 \in [0,T]$ and $\bm{x}_1, \bm{x}_2 \in \bR^n$, let $(\bm{X}^j_t,\bm{Y}^j_t,\bm{Z}^j_t)$ be the short notation\footnote{The notation here is slightly inconsistent with the statement in Section~\ref{sec_mathformulation}: A boldface character with a superscript $j$ denotes a solution with the $j^{th}$ initial condition $(t_0, \bx_j)$.}  for $(\bm{X}^{t_0,\bm{x}_j}_t,\bm{Y}^{t_0,\bm{x}_j}_t,\bm{Z}^{t_0,\bm{x}_j}_t)$, $j = 1,2, t\in[t_0,T]$. The wellposedness of BSDEs \eqref{lem1-BSDE} and \eqref{lem1-BSDE2} is classical in literature; see, {\it e.g.}, \cite{PaPe:90}.

We define $\delta \bm{X}_t$, $\delta \bm{Y}_t$, $\delta \bm{Z}_t$, $\delta \bm{F}_t$ and $\delta \bm{\Sigma}_t$ as follows:
\begin{align*}
    &\delta \bm{X}_t = \bm{X}^1_t - \bm{X}^2_t, \quad \delta \bm{Y}_t = \bm{Y}^1_t - \bm{Y}^2_t,\quad \delta \bm{Z}_t = \bm{Z}^1_t - \bm{Z}^2_t, \\
    &\delta  \bm{F}_t = \bm{F}(t,\bm{X}_t^1,\bm{Z}_t^1) - \bm F(t,\bm{X}_t^2,\bm{Z}_t^2),\quad \delta \Sigma_t = \Sigma(t,\bm{X}^1_t) - \Sigma(t,\bm{X}^2_t).
\end{align*}
Then, we have
\begin{align}
    \rmd \delta \bm{X}_t &= \delta \Sigma_t \rmd \bm{W}_t, \\
    \rmd \delta \bm{Y}_t &= - \delta \bm{F}_t \rmd t + (\delta \bm{Z}_t)\transpose  \rmd \bm{W}_t,
\end{align}
and It\^{o}'s lemma gives 
\begin{align}
    \rmd |\delta \bm{X}_t|^2 &= \|\delta \Sigma_t\|_F^2 \rmd t + 2(\delta \bm{X}_t)\transpose  \delta \Sigma_t \rmd \bm{W}_t, \\
    \rmd |\delta \bm{Y}_t|^2 &= \left[-2\delta \bm{F}_t\cdot\delta \bm{Y}_t + \|\delta \bm{Z}_t\|_F^2\right] \rmd t + 2(\delta \bm{Z}_t \delta\bm{Y}_t)\transpose\rmd \bm{W}_t.
\end{align}
Taking the expectation on both sides yields
\begin{equation}
    \EE|\delta \bm{X}_t|^2 = |\bm{x}_1 - \bm{x}_2|^2 + \int_{t_0}^{t}\EE\|\delta \Sigma_s\|_F^2 \rmd s \le |\bm{x}_1-\bm{x}_2|^2 + \overline{L} \int_{t_0}^{t}\EE|\delta \bm{X}_s|^2 \rmd s,
\end{equation}
and by Gr\"onwall's inequality, we have $\EE|\delta \bm{X}_t|^2 \le e^{\overline{L} (t-t_0)}|\bm{x}_1 - \bm{x}_2|^2$. Similarly, we deduce that
\begin{align}
    \EE|\delta \bm{Y}_t|^2 &= \EE|\delta \bm{Y}_T|^2 + \int_{t}^T \EE\left[2\delta \bm{F}_s \cdot \delta \bm{Y}_s - \|\delta \bm{Z}_s\|_F^2\right] \rmd s \notag\\
    &\le g_x \EE|\delta \bm{X}_T|^2 +  \int_{t}^T \left\{\overline{L} \EE|\delta \bm{Y}_s|^2 + \overline{L}^{-1}\EE|\delta \bm{F}_s|^2 - \EE\|\delta \bm{Z}_s\|_F^2 \right\}\rmd s \notag\\
    &\le g_x \EE|\delta \bm{X}_T|^2 +  \int_{t}^T \left\{\overline{L} \EE|\delta \bm{Y}_s|^2 + \overline{L}^{-1}\left[\overline{L} \EE|\delta \bm{X}_s|^2 + \overline{L} \EE\|\delta \bm{Z}_s\|_F^2\right] - \EE\|\delta \bm{Z}_s\|_F^2 \right\}\rmd s \notag\\
    &\le (g_x + T - t_0) e^{\overline{L}(T-t_0)}|\bm{x}_1 - \bm{x}_2|^2 + \overline{L} \int_{t}^{T}\EE|\delta \bm{Y}_s|^2 \rmd s 
\end{align}
and by Gr\"onwall's inequality, we have
\begin{equation}\label{eq:dyto}
|\delta\bm{Y}_{t_0}|^2 \le (g_x + T - t_0)e^{2\overline{L}(T-t_0)}|\bm{x}_1 - \bm{x}_2|^2.
\end{equation}

Following the argument in \cite[Theorem~3.1]{ma2002representation}, we define $\bm u(t,\bx) = \bm{Y}_t^{t,\bx}$ and deduce $|\bm u(t_0,\bm{x}_1) - \bm u(t_0,\bm{x}_2)|^2 \le (g_x+T)e^{2\overline{L}T} |\bm{x}_1 - \bm{x}_2|^2$ from \eqref{eq:dyto}. Therefore, we claim
\begin{equation*}
    \|\nabla_{\bx} \bm u(t,\bx)\|_S^2 \le (g_x+T)e^{2\overline{L}T} \text{ a.s. with the Lebesgue measure on } \bR^n,  \quad \forall t \in [0,T].
\end{equation*}
Also noticing that $\bm{Z}_t = (\Sigma \transpose \nabla_{\bx} \bm u)(t,\bm{X}_t)$ $\mathbb{P}$-a.s. (cf. \cite[Theorem~3.1]{ma2002representation}),
\begin{equation*}
    \|(\Sigma \transpose \nabla_{\bx} \bm u)(t,\bx)\|_S^2 \le \|\Sigma(t,\bx)\|_S^2 \|\nabla_{\bx} \bm u(t,\bx)\|_S^2 \le  \overline{M}(g_x+T)e^{2\overline{L}T},
\end{equation*}
and the law of $\bm{X}_t$ is absolute continuous with respect to the Lebesgue measure on $\bR^n$, we can get
\begin{equation*}
    \|\bm{Z}_t\|_S^2 \le \overline{M} (g_x+T)e^{2\overline{L}T} \quad \mathbb{P}\text{-a.s.}, \quad \forall t \in [0,T].
\end{equation*}
Finally, since $\bm{Z}_t$ is measurable with respect to $\mathcal{L}([0,T])\times \mathcal{F}$, where $\mathcal{L}([0,T])$ denotes the collection of Lebesgue measurable sets on $[0,T]$, we obtain our conclusion.
\end{proof}

Using the above lemma, we can prove that the solution to the BSDE system (\ref{def_BSDE}) has a bounded $Z$-component, for sufficiently small $T$. Note that the standard estimate can not be applied directly here, as the driver $\miniH$ defined in \eqref{def_BSDE} is not global Lipschitz in $\bm p $.

\begin{proof}[Proof of Proposition \ref{prop:assump3_smalltime}]
We first prove the existence. Fix $M_z > 0$, we use $\mathrm{P}_{M_z}(\bZ)$ to denote the projection from $\bR^{k\times N}$ to $\{\bZ \in \bR^{k\times N}: \|\bZ\|_S^2 \le M_z\}$, and use $\mathrm{P}_{M_z}(\bZ)^i$ for its $i^{th}$ column. Let $\Tilde{\bm{H}} = [\Tilde{H}^1, \ldots, \Tilde{H}^N]\transpose$ with $\Tilde{H}^i(t,\bm{x},\bm{p}) = \phi(t,\bm{x},\bm{\alpha}(t,\bm{x},\bm{p}))\cdot\mathrm{P}_{M_z}(\bm{p})^i + f^{i}(t,\bm{x},\bm{\alpha}(t,\bm{x},\bm{p}))$,
then its Lipschitz constants with respect to $\bm x$ and $\bm p$ are computed by 
\begin{align}\label{H_Lipschitz}
    &\quad~|\Tilde{\bm{H}}(t,\bm{x}_1,\bm{p}_1) - \Tilde{\bm{H}}(t,\bm{x}_2,\bm{p}_2)|^2  \notag \\
    &\le 3M_z|\phi(t,\bm{x}_1,\bm{\alpha}(t,\bm{x}_1,\bm{p}_1)) - \phi(t,\bm{x}_2,\bm{\alpha}(t,\bm{x}_2,\bm{p}_2))|^2+3M \|\mathrm{P}_{M_z}(\bm{p}_1) - \mathrm{P}_{M_z}(\bm{p}_2)\|_F^2\notag \\
    &\quad +3|\bm{f}(t,\bm{x}_1,\bm{\alpha}(t,\bm{x}_1,\bm{p}_1))- \bm{f}(t,\bm{x}_2,\bm{\alpha}(t,\bm{x}_2,\bm{p}_2))|^2 \notag \\
    & \le 3(M_z L + L)|\bm{x}_1-\bm{x}_2|^2 + 3M\|\bm{p}_1-\bm{p}_2\|_F^2 +3(M_z L + L)|\bm{\alpha}(t,\bm{x}_1,\bm{p}_1)-\bm{\alpha}(t,\bm{x}_2,\bm{p}_2)|^2 \notag \\
    & \le 3(M_z + 1)(L^2+L)|\bm{x}_1 - \bm{x}_2|^2 + 3[M + (M_z L + 1)L]\|\bm{p}_1 - \bm{p}_2\|_F^2.  
\end{align}
Now define $M_z = 2M (L+1)$  and consider the solution $(\bm{X}_t,\Tilde{\bm{Y}}_t,\Tilde{\bm{Z}}_t)$ to the following BSDE system
\begin{equation}
    \begin{dcases}
    \bm{X}_t = \bm{x}_0 + \int_{0}^{t}\Sigma(s,\bm{X}_s)\rmd \bm{W}_s, \\
    \Tilde{\bm{Y}}_t = \bm{g}(\bm{X}_T) + \int_{t}^T \Tilde{\bm{H}}(s,\bm{X}_s,\Tilde{\bm{Z}}_s)\rmd s - \int_{t}^T(\Tilde{\bm{Z}}_s)\transpose \rmd \bm{W}_s.
    \end{dcases}
\end{equation}
Following Lemma~\ref{lem1}, Assumption \ref{assumption 1 revise} and inequality \eqref{H_Lipschitz} with $g_x = L$, $\overline{L} = 3(2ML+2M+1)(L^2+L)$  and $\overline{M} = M$, we have:
\begin{align}
    \|\Tilde{\bm {Z}}_t\|^2_S &\le M(L +T) e^{6(2ML+2M+1)(L^2+L)T} \\
    & := M(T).
\end{align}
Therefore, if $M(T) \le 2 M(L+1)$,  $(\bm{X}_t,\Tilde{\bm{Y}}_t,\Tilde{\bm{Z}}_t)$ is the desired solution to the BSDE system (\ref{def_BSDE}) with $\|\Tilde{\bm{Z}}_t\|_S^2 \le 2 M(L+1)$, which can be fulfilled if $T$ is small enough.
\end{proof}
\begin{proof}[Proof of Proposition \ref{prop:unique}]
 Let $(\bm{X}_t,\bm{Y}'_t,\bm{Z}'_t)$ be another adapted solution of the BSDE system \eqref{def_BSDE} such that
\begin{equation}
    \EE \left[\sup_{0\le t \le T}(|\bm{X}_t|^2 +|\bm{Y}'_t|^2) + \int_{0}^T \|\bm{Z}'_t\|_F^2 \rmd t\right] < + \infty.
\end{equation}
Define
$\delta \bm{Y}_t = \bm{Y}'_t - \tilde{\bm{Y}}_t$, $\delta \bm{Z}_t = \bm{Z}'_t - \tilde{\bm{Z}}_t$ and $\delta \bm{H}_t = \miniH(t,\bm{X}_t,\bm{Z}'_t) - \miniH(t,\bm{X}_t,\tilde{\bm{Z}_t})$ and we can write 
\begin{equation}\label{thm2_eq1}
    \rmd \delta \bm{Y}_t = - \delta \bm{H}_t\rmd t + \delta \bm{Z}_t \rmd \bm{W}_t.
\end{equation}
It\^{o}'s lemma gives 
\begin{equation}
    \ud|\delta \bm{Y}_t|^2 = [-2\delta \bm{Y}_t \cdot \delta \bm H_t + \|\delta \bm{Z}_t\|_F^2]\rmd t +2 (\delta \bm{Z}_t\delta \bm{Y}_t) \transpose \rmd \bm{W}_t.
\end{equation}
Taking expectation on both side and using $\delta \bm{Y}_T = 0$, we deduce that
\begin{equation}\label{thm2_eq2}
    \EE|\delta \bm{Y}_t|^2 + \int_{t}^{T}\EE\|\delta \bm{Z}_s\|_F^2 \rmd s = 2\int_{t}^{T}\EE[\delta \bm{H}_s \cdot\delta \bm{Y}_s]\rmd s \le \lambda \int_{t}^{T}\EE|\delta \bm{Y}_s|^2\rmd s + \lambda^{-1} \int_{t}^{T}\EE|\delta \bm{H}_s|^2 \rmd s,
\end{equation}
for any $\lambda > 0$.
By
\begin{align}
    \delta \bm{H}_t &= \phi(t,\bm{X}_t,\bm{\alpha}(t,\bm{X}_t,\bm{Z}'_t))\cdot \bm{Z}'_t - \phi(t,\bm{X}_t,\bm{\alpha}(t,\bm{X}_t,\tilde{\bm{Z}}_t)\cdot \tilde{\bm{Z}}_t  \notag \\
    &\quad + \bm{f}(t,\bm{X}_t,\bm{\alpha}(t,\bm{X}_t,\bm{Z}'_t))-\bm{f}(t,\bm{X}_t,\bm{\alpha}(t,\bm{X}_t,\tilde{\bm{Z}}_t))\notag \\
    &=\phi(t,\bm{X}_t,\bm{\alpha}(t,\bm{X}_t,\bm{Z}'_t)) \cdot (\bm{Z}'_t - \tilde{\bm{Z}}_t) + [\phi(t,\bm{X}_t,\bm{\alpha}(t,\bm{X}_t,\bm{Z}'_t)) - \phi(t,\bm{X}_t,\bm{\alpha}(t,\bm{X}_t,\tilde{\bm{Z}}_t)]\cdot \tilde{\bm{Z}}_t\notag \\
    &\quad + \bm{f}(t,\bm{X}_t,\bm{\alpha}(t,\bm{X}_t,\bm{Z}'_t))-\bm{f}(t,\bm{X}_t,\bm{\alpha}(t,\bm{X}_t,\tilde{\bm{Z}}_t)),
\end{align}
we have $|\delta \bm{H}_t|^2 \le L_z \|\delta \bZ_t\|_F^2$ with $L_z = 3[M + M'L^2+ L^2]$. 
Taking $\lambda = L_z$ in \eqref{thm2_eq2}, we deduce that  $\EE|\delta \bm{Y}_t|^2 \le L_z\int_{t}^{T}\EE|\delta \bm{Y}_s|^2 \rmd s$ and therefore $\bm{Y}_t' \equiv \tilde{\bm{Y}}_t$ by Gr\"onwall's inequality. 
We then have $\int_{0}^{T}\EE\|\delta \bm{Z}_t\|_F^2 \rmd t = 0$ from the first equality in \eqref{thm2_eq2}, which implies $\bm{Z}_t' \equiv \tilde{\bm{Z}}_t$.
\end{proof}

\end{document}